\documentclass[11pt]{amsart}
\usepackage{amssymb,amsmath, color}

\theoremstyle{plain}
\newtheorem{thm}{Theorem}[section]

\newtheorem{lemma}[thm]{Lemma}
\newtheorem{cor}[thm]{Corollary}
\newtheorem{prop}[thm]{Proposition}

\theoremstyle{definition}
\newtheorem{defn}[thm]{Definition}

\theoremstyle{remark}

\newtheorem{rem}[thm]{Remark}

\newtheorem*{remark*}{Remark}

\numberwithin{equation}{section}

\newcommand{\xb}{{\bf x}}

        \newcommand{\field}[1]{{\mathbb{#1}}}

        \newcommand{\RR}{\field{R}}

\allowdisplaybreaks

\begin{document}
\title[Accurate semiclassical spectral asymptotics]{Accurate semiclassical spectral asymptotics for a
two-dimensional magnetic Schr\"odinger operator}

\author{Bernard Helffer}

\address{D\'epartement de Math\'ematiques, B\^atiment 425, Univ
Paris-Sud et CNRS, F-91405 Orsay C\'edex, France}
\email{Bernard.Helffer@math.u-psud.fr}

\author{Yuri A. Kordyukov }
\address{Institute of Mathematics, Russian Academy of Sciences, 112 Chernyshevsky
str. 450008 Ufa, Russia} \email{yurikor@matem.anrb.ru}

\thanks{B.H. is partially supported by INSMI
CNRS and by the ANR programme Nosevol. Y.K. is partially supported
by the Russian Foundation of Basic Research, projects 12-01-00519-a
and 13-01-91052-NCNI-a, and the Ministry of education and science of
Russia, project 14.B37.21.0358. }


\begin{abstract}
We revisit the problem of semiclassical spectral asymptotics for a
pure magnetic Schr\"odinger operator on a two-dimensional Riemannian
manifold. We suppose that the minimal value $b_0$ of the intensity
of the magnetic field is strictly positive, and the corresponding
minimum is unique and non-degenerate. The purpose is to get the
control on the spectrum in an interval $(hb_0, h(b_0 +\gamma_0)]$
for some $\gamma_0>0$ independent of the semiclassical parameter
$h$. The previous papers by Helffer-Mohamed and by Helffer-Kordyukov
were only treating the ground-state energy or a finite (independent
of $h$) number of eigenvalues.  Note also that  N. Raymond and S.~Vu
Ngoc have recently developed a different approach of the same
problem.
\end{abstract}
 \maketitle

%\tableofcontents

 \section{Introduction and main results}
Let $ M$ be a compact connected oriented manifold of dimension
$n\geq 2$ (possibly with boundary).  Let $g$ be a $C^1$
Riemannian metric and ${\bf B}\in C(M,\Lambda^2 T^*M)$ a continuous
real-valued closed 2-form on $M$. Assume that $\bf B$ is exact and
choose a real-valued 1-form ${\bf A}\in C^1(M,\Lambda^1 T^*M)$ on
$M$ such that $d{\bf A} = \bf B$. Thus, one has a natural mapping
\[
u\mapsto ih\,du+{\bf A}u
\]
from $C^1_c(M)$ to the space $C(M,\Lambda^1 T^*M)$ of continuous,
compactly supported one-forms on $M$ and  from $C^2_c(M)$ to the
space $C^1(M,\Lambda^1 T^*M)$. The Riemannian metric allows us to
define scalar products in these spaces and consider the adjoint
operator
\[
(ih\,d+{\bf A})^* : C^1(M,\Lambda^1 T^*M) \to C_c(M)\,.
\]
A Schr\"odinger operator with magnetic potential $\bf A$ is defined
on $C^2_c(M)$ by the formula
\[
H^h = (ih\,d+{\bf A})^* (ih\,d+{\bf A})\,.
\]
Here $h>0$ is a semiclassical parameter, which is assumed to be
small. If $M$ has a non-empty boundary, we will assume that the
operator $H^h$ satisfies the Dirichlet boundary conditions.

We are interested in the semiclassical asymptotics of the low-lying
eigenvalues of the operator $H^h$. This problem was studied in
\cite{Fournais-Helffer:book,luminy,miniwells,HKI,HM,HM01,HelSj7,Mat,MatUeki,Mont}
 (see also references therein).

We suppose that $M$ is two-dimensional. Then we can write ${\bf
B}=b\, d\xb_g$, where $b\in C(M)$ and $d\xb_g$ is the Riemannian
volume form. Let
\[
b_0=\min_{\xb \in M} |b(\xb )|\,.
\]
We assume that:
\begin{itemize}
\item $b_0>0\,$;
\item the set $\{\xb \in M : |b(\xb)| =b_0\}$ is a point  $\xb_0\,$, which is contained in the interior of $M\,$;
\item  $b$ is $C^\infty$ in a neighborhood of $\xb_0$ and there is a constant $C>0$ such that for all $\xb$
in some neighborhood of $\xb_0$ the estimates hold:
\[
C^{-1}d(\xb,\xb_0)^2\leq |b(\xb)|-b_0 \leq C d(\xb,\xb_0)^2\,.
\]
\end{itemize}
We introduce
\[
a={\rm Tr}\left(\frac12 {\rm Hess}\,b(\xb_0)\right)^{1/2}, \quad
d=\det \left(\frac12 {\rm Hess}\,b(\xb_0)\right)\,,
\]
and denote by $\lambda_0(H^h)\leq \lambda_1(H^h)\leq \lambda_2(H^h)\leq
\ldots$ the eigenvalues of the operator $H^h$ in $L^2(M)\,$.

The following theorem proved in \cite{HM01,HKI} gives upper and
lower estimates of $\lambda_j(H^h)$ as $h\to 0$. The contribution of
\cite{HKI} improves the result of \cite{HM01} which only gives a
two-terms asymptotics for the ground state energy in the flat case.

\begin{thm}\label{t:main}
Under current assumptions, for any  $j\in \mathbb N$, there exist $C_j>0$
and $h_j>0$ such that, for any $h\in (0,h_j]$,
\begin{multline*}
hb_0 +h^2\left[\frac{2d^{1/2}}{b_0}j+
\frac{a^2}{2b_0}\right]-C_jh^{19/8}\leq \lambda_j(H^h)\\ \leq hb_0
+h^2\left[\frac{2d^{1/2}}{b_0}j+ \frac{a^2}{2b_0}\right]+C_jh^{5/2}.
\end{multline*}
\end{thm}

The main purpose of this paper is to reinterpret and extend
Theorem~\ref{t:main} in the following setting. We will consider the
magnetic Schr\"odinger operator $H^h$ in the flat Euclidean space
${\mathbb R}^2$:
\begin{equation} \label{e:Hh}
H^h=h^2D_x^2+(hD_y+A(x,y))^2.
\end{equation}
The magnetic field $\bf B$ is given by
\[
{\bf B}=b\,dx\wedge dy\,\mbox{ with }\quad b(x,y) =\frac{\partial
A}{\partial x}(x,y)\,.
\]
Let
\[
b_0=\min_{(x,y)\in {\mathbb R}^2} |b(x,y)|\,.
\]
We assume
\[
|b(x,y)| <  b_0+\eta_0:=\liminf_{|x|+|y|\to \infty} |b(x,y)|, \quad
\eta_0>0.
\]
One can prove (see Theorem~\ref{t:discrete-spectrum}) that, for any
$\eta_1<\eta_0$, there exists $h_1>0$ such that
\[
\sigma (H^h)\cap [0, h(b_0+\eta_1)) \subset \sigma_d (H^h)\,, \quad
\forall h\in (0,h_1]\,.
\]
As above, we assume that:
\begin{itemize}
\item $b_0>0$;
\item the set $\{(x,y)\in {\mathbb R}^2 : |b(x,y)| =b_0\}$ is a single point $(x_0,y_0)\,$;
\item  $b$ is $C^\infty$ in a neighborhood of $(x_0,y_0)$, and
$(x_0,y_0)$ is a non-degenerate minimum:
\[
\operatorname{Hess} b (x_0,y_0)>0\,.
\]
\end{itemize}

We take linear coordinates in ${\mathbb R}^2$ such that
$(x_0,y_0)=(0,0)\,$. We can also assume after possibly a gauge
transformation that:
\begin{equation}\label{A00}
A(0,0)=0 \mbox{ and } \frac{\partial A}{\partial y} (0,0)=0\,.
\end{equation}
We have a diffeomorphism $\phi : {\mathbb R}^2\to {\mathbb R}^2$
defined by
\[
\phi(x,y)=(A(x,y),y), \quad (x,y)\in {\mathbb R}^2\,.
\]
We then associate with $b$  a function $\hat b\in C^\infty({\mathbb
R}^2)$ by
\[
\hat b = b\circ \phi^{-1}\,.
\]

Our goal is to prove the following theorem.

\begin{thm}\label{mainth}
 There exist $h_0>0, \epsilon_0>0, \gamma_0 \in (0,\eta_0)$,  $h\mapsto \gamma_0(h)$ defined for
   $(0,h_0]$ such that $\gamma_0(h)\to \gamma_0$ as $h\to 0$,
and a semiclassical symbol $p_{\rm eff}(y, \eta ,h, z)$, which is
defined in a neighborhood $\Omega\subset {\mathbb R}^2$ of the set
$\{(y,\eta)\in \mathbb R^2: \hat b(y,\eta) \leq b_0+\gamma_0\}$ for
$h\in (0,h_0]$ and  $z\in \mathbb C$ such that
$|z|<\gamma_0+\epsilon_0$, of the form
\begin{equation}\label{bth}
p_{\rm eff} (y, \eta ,h, z) \sim\sum_{j\in \mathbb N} p_{\rm
eff}^{j}(y,\eta,z) h^j\,,
\end{equation}
with
\begin{equation}\label{cth}
 p_{\rm eff}^0 (y,\eta,z) =  \hat b(y,\eta) - b_0 -z\,,
\end{equation}
such that $\lambda_h \in \sigma (H^h)\cap [0, h (b_0
+\gamma_0(h)))$,  if and only if the associated
$h$-pseudo\-diffe\-ren\-tial operator $p_{\rm eff}(y, hD_y,h, z(h))$
has an approximate $0$-ei\-gen\-fun\-cti\-on $u_h^{qm}\in
C^\infty(\mathbb R)$, i.e.
\begin{equation}\label{ath}
p_{\rm eff}(y, hD_y,h, z(h)) u_h^{qm} = \mathcal O (h^\infty)\,,
\end{equation}
with
\[
z(h)= \frac 1 h (\lambda_h - h b_0)+\mathcal O (h^\infty)\,,
\]
$|z(h)|<\gamma_0(h)$ for any $h\in (0,h_0]$, and such that the
frequency set of $u_h^{qm}$ is non-empty and contained in
$\Omega\,$.
\end{thm}

\begin{rem}
Here \eqref{ath} makes sense modulo $\mathcal O (h^\infty)$ by
extending first the symbol $p_{\rm eff}(y, \eta ,h, z)$ outside the
neighborhood $\Omega$ to a semiclassical symbol in $\mathbb R^2$ and
defining then the operators $p_{\rm eff}(y, hD_y,h, z)$ by the Weyl
calculus. Using the localization of the frequency set of $u_h^{qm}$,
the left hand side of \eqref{ath} does not depend on the extension
up to an error which is $\mathcal O (h^\infty)$.
\end{rem}

\begin{rem}
For any $E\in [b_0,b_0+\gamma_0)$, the spectrum of the operator
$H^h$ (divided by $h$) is determined near  $E$ (say in an interval
$(E-C h^\frac 12, E+ C h^\frac 12)$) and modulo $\mathcal O(h^\frac
32)$ by the spectrum of $\hat b(y,hD_y) + h b_1(y, hD_y,E)$, where
one can use the Bohr-Sommerfeld rule (see \cite{HelRo} or
\cite{HeSjharp1}  for a mathematical justification) for determining
the energy levels. \end{rem}

\begin{rem} Of course $\gamma_0$ is such that $b(x,y)$ is
$C^\infty$ in a neighborhood of $b^{-1}((0, b_0+\gamma_0])$.
\end{rem}

Denote by $\lambda_0(H^h)\leq \lambda_1(H^h)\leq \lambda_2(H^h)\leq
\ldots$ the eigenvalues of the operator $H^h$ in $[0, h(b_0
+\gamma_0))$.

\begin{thm}\label{t:mainbis}
Under current assumptions, for any $j\in \mathbb N$, there exists a
sequence $(\alpha_{j,\ell})_{\ell\in \mathbb N}$ such that
\[
\lambda_j(H^h) \sim h \sum_{\ell=0}^\infty \alpha_{j,\ell}
h^{\ell}\,.
\]
In other words, for any $N$, there exist $C_{j,N}>0$ and $h_{j,N}>0$
such that, for any $h\in (0,h_{j,N}]$,
\[
|\lambda_j(H^h) - h \sum_{\ell=0}^N \alpha_{j,\ell} h^{\ell }|\leq
C_{j,N} h^{N+2}\,.
\]
\end{thm}

By the results of \cite{HM01} (see Theorem~\ref{t:main}), it
follows that
\[
\alpha_{j,0} = b_0, \quad \alpha_{j,1}= \frac{2d^{1/2}}{b_0}j+
\frac{a^2}{2b_0}\,.
\]
In \cite{HKI}, it was shown that, in the case of magnetic
Schr\"odinger operator on a two-dimensional Riemannian manifold,
each $\lambda_j$ admits an asymptotic expansion in the form
\[
\lambda_j(H^h) \sim h \sum_{\ell=0}^\infty \alpha_{j,\ell/2}
h^{\ell/2}\,.
\]
Theorem~\ref{t:mainbis} improves this result in the flat case,
showing that no odd powers of $h^{1/2}$ actually occur. It is also
proved in \cite{RV}.

\begin{cor} There exists $\gamma_0 \in (0, \eta_0)$, $h_0>0$ and $C>0$ such that
$$
\lambda_{j+1}(H^h) -\lambda_j(H^h) \geq \frac 1C h^2 \,,\, \forall
h\in (0,h_0]\,,
$$
for any $j$ such that $\lambda_{j+1}(H^h) < h (b_0 + \gamma_0)$.
\end{cor}

The proof of Theorem~\ref{mainth} is based on Grushin's method. As
the name ``Grushin's method'' indicates, the technique comes back to
Grushin \cite{Gru}. It was popularized by J. Sj\"ostrand starting
from 1974 \cite{Sj}. The method turned out to be very effective not
only in hypoellipticity theory \cite{Sj}, \cite{He77}, but also in
spectral theory \cite{HelSjSonder}, \cite{HeSjharp3}. The reader can
find a nice presentation of this method in \cite{SjZw}.

For the proof, we first make some changes of variables and
asymptotic expansions to put the operator $H^h$ in a normal form
near $(0,0)$. Then we construct an appropriate Grushin problem in a
neighborhood of $(0,0)$ and apply Grushin's method. This approach is
local near the minimum point $(0,0)$.

From a close but different point of view, the problem under
consideration was studied by N. Raymond and S. Vu Ngoc \cite{RV}.
Their proof is reminiscent of Ivrii's approach (see \cite{Iv0} or in his
book \cite{Iv} in different versions Chapter 18) and uses a Birkhoff normal
form. This approach has the advantage to be semi-global and uses more general
symplectomorphisms and their quantizations.

Theorem 1.5 in \cite{RV} is stronger than our Theorem~\ref{mainth}
because Theorem~\ref{mainth} gives a description of the spectrum of
$H^h$ in the interval $[hb_0,h(b_0+\gamma_0))$ for some $\gamma_0\in
(0,\eta_0)$, whereas, in \cite[Theorem 1.5]{RV}, $\gamma_0\in
(0,\eta_0)$ is arbitrary. On the other hand, the symbol of the
effective Hamiltonian in \cite[Theorem 1.5]{RV} seems to be less
explicit than in Theorem~\ref{mainth} (see in Section \ref{s7}). The
other point could be that our approach allows us to treat an
additional term $h^2 V(x,y)$. This will complete the analysis of
Helffer-Sj\"ostrand \cite{HelSjSonder} in the case of the constant
magnetic field (strong magnetic case). This  kind of approach
appears also in \cite{Kar} (see Remark~3.1 and more specifically
(3.5)). The case with an additional term of the form $h V$ could
also be interesting. Note that, in the book \cite{Iv} (see also the
announcement in \cite{Iv0}), there is also an interesting normal
form corresponding to the case of dimension 3. We hope to come back
to this point in a near future.

The paper is organized as follows. In Section~\ref{s:prelim}, we
establish some general properties of the magnetic Schr\"odinger
operator in the flat Euclidean space ${\mathbb R}^2$. First, we
recall the proof  that the spectrum of $H^h$ on the interval $[hb_0,
h(b_0+\eta_0))$ is discrete. Then we show that, if $A$ is changed at
infinity in such a way that a neighborhood of $b^{-1}([b_0,
b_0+\eta_0])$ is unchanged, then this change will only affect the
spectrum of $H^h$ on the interval $[hb_0, h(b_0+\gamma_0)]$ with
$\gamma_0<\eta_0$ by exponentially small corrections. This fact
allows us to impose rather strong assumptions on $A$ in our further
considerations. More precisely, we will assume that, for any
$(k,\ell)\in \mathbb N^2$ with $k+\ell>0\,$, the derivative
$\partial^{k}_x\partial^\ell_yA$ is uniformly bounded in $\mathbb
R^2$. In Section~\ref{s:normalforms}, using some changes of
variables and asymptotic expansions, we put the operator into a
normal form. Section~\ref{s:eigenfunctions} is devoted to some
asymptotic properties of eigenfunctions associated with the spectrum
of $H^h$ on the interval $[hb_0, h(b_0+\eta_0)]$. First, we obtain
an information on the frequency set of these eigenfunctions. Then we
derive estimates for such eigenfunctions in the Sobolev spaces
$B^k(\mathbb R^2)$ defined in \eqref{e:defBk}.
Section~\ref{s:Grushin} is devoted to construction and investigation
of an appropriate Grushin problem. In
Section~\ref{s:Grushin-quasimodes}, we complete the proof of
Theorem~\ref{mainth}. Finally we discuss in Section \ref{s7}
possible extensions.

\section{Preliminaries on the magnetic Schr\"odinger
operator}\label{s:prelim} In this section, we will discuss some
general properties of the magnetic Schr\"odinger operator $H^h_A$ in
the flat Euclidean space ${\mathbb R}^2$ given by
\[
H^h_A=h^2D_x^2+(hD_y+A(x,y))^2\,,
\]
where $A\in C^1(\mathbb R^2)\,$.\footnote{The optimal condition of
regularity of $A$ is weaker but our assumption permits to avoid
writing ``almost everywhere'' in the condition.}

Under these assumptions, the operator $H^h_A$ is essentially
self-adjoint in $L^2(\mathbb R^2)$ with initial domain
$C^\infty_c(\mathbb R^2)$ (see, for instance, \cite[Theorem
1.2.2]{Fournais-Helffer:book}).

Let
\[
b_0=\min_{(x,y)\in {\mathbb R}^2} |b(x,y)|\,.
\]
We assume
\begin{equation}\label{e:def-eta0}
|b(x,y)| <  b_0+\eta_0:=\liminf_{|x|+|y|\to \infty} |b(x,y)|\,, \quad
\eta_0>0\,.
\end{equation}
For any self-adjoint operator $P$ in a Hilbert space $\mathcal H$,
we denote by $\sigma(P)$ the spectrum of $P$, by $\sigma_d(P)$ the
discrete spectrum of $P$ and by $\sigma_{\rm ess}(P)$ the essential
spectrum of $P\,$.

\begin{thm}\label{t:discrete-spectrum}
For any $\eta_1<\eta_0$, we have
\[
\sigma (H^h_A)\cap [0, h(b_0+\eta_1)) \subset \sigma_d (H^h_A)\,,
\quad \forall h>0\,.
\]
\end{thm}

\begin{proof}
We recall the estimate
\begin{equation}\label{e:lower-bound}
(H^h_Au,u)\geq h\int_{\mathbb R^2} |b(x,y)| |u(x,y)|^2\,dx\,dy\,,
\quad \forall u \in H^1_0(\mathbb R^2)\,.
\end{equation}
The theorem follows immediately from this estimate and Persson's
characterization of the bottom of the essential spectrum of a
self-adjoint uniformly elliptic operator $P$ on $L^2(\mathbb R^n)$
(see \cite{Agmon}),
\[
{\rm Inf}\,\sigma_{\rm ess}(P)=\lim_{R\to\infty}
\inf_{\substack{u\in D(P), \|u\|\neq 0, \\ {\rm supp}\,u\subset
\{|x|\geq R\}}} \frac{(Pu,u)}{\|u\|^2}\,.
\]
\end{proof}

\begin{thm}\label{t:comparison}
Let $\eta \in (0,\eta_0)$ with $\eta_0$ defined in
\eqref{e:def-eta0}. Assume that $\tilde A\in C^1(\mathbb R^2)$ is
such that $\tilde A(x,y)=A(x,y)$ for any $(x,y)\in \mathbb R^2$ such
that $b(x,y)<b_0+\eta$. For any $\eta_1$ such that $0<\eta_1<\eta$,
there exist $h_0>0$, $\alpha>0$, a map $(0,h_0]\ni h\mapsto
\epsilon(h)>0$, such that $\epsilon(h)\to 0$ as $h\to 0$ and, for
any $h\in (0,h_0]$, a one to one map
\[
\Psi^h : \sigma (H^h_{\tilde A})\cap [hb_0,
h(b_0+\eta_1+\epsilon(h))]\to \sigma (H^h_A)\cap [hb_0,
h(b_0+\eta_1+\epsilon(h))]\,,
\]
such that
\[
|\Psi^h(\lambda)-\lambda|<e^{-\frac{\alpha}{h^{1/2}}}, \quad \forall
h\in (0,h_0]\,.
\]
\end{thm}

\begin{proof}
We choose a function $\epsilon(h)$ defined for any $h\in (0,h_0]$
with some $h_0>0$ so that there exist a function $a(h)$, $h\in
(0,h_0]$, such that
\[
a(h)=\mathcal O\left(h^{N_0}\right)
\]
with some $N_0$ and, for any $h\in (0,h_0]$, we have
\begin{gather*}
\sigma (H^h_A)\cap [h(b_0+\eta_1+\epsilon(h)),
h(b_0+\eta_1+\epsilon(h))+a(h))
=\emptyset\,,\\
\sigma (H^h_{\tilde A})\cap [h(b_0+\eta_1+\epsilon(h))\,,
h(b_0+\eta_1+\epsilon(h))+a(h))=\emptyset\,.
\end{gather*}

To see the existence of such $\epsilon(h)$ and $a(h)$, we use a
polynomial upper estimate for the number of eigenvalues of operators
$H^h_A$ and $H^h_{\tilde A}$. Denote by $N(H^h_A, \lambda)$ the
number of eigenvalues of the operator $H^h_A$ less than or equal to
$\lambda$. By \cite[Lemma~4.6]{RV}, for any $C_1<b_0+\eta_0$ there
exists $C>0$ such that for all $h>0$, we have
\begin{equation}\label{e:N}
N(H^h_A, C_1h)\leq Ch^{-1}.
\end{equation}
Now we just take the interval $[h(b_0+\eta_1), h(b_0+\eta_1+ch))$
with some $c$, divide it into the union of disjoint intervals of the
same length $h^4$, and, using \eqref{e:N}, immediately get the
existence of an interval of the above form free of eigenvalues for
any $h>0$ small enough.

Let $E$ (resp. $\tilde E$) be the eigenspace of $A$ (resp. $\tilde
A$) associated with $\sigma(H^h_A)\cap [hb_0,
h(b_0+\eta_1+\epsilon(h))]$ (resp. $\sigma(H^h_{\tilde A})\cap
[hb_0, h(b_0+\eta_1+\epsilon(h))]$). By \eqref{e:N},  we have
\[
\dim E =\mathcal O(h^{-1})\,,\quad \dim \tilde E =\mathcal O(h^{-1})\,,
\quad h\to 0.
\]

Let $u_h$ be an eigenfunction of $H^h_A$ with the corresponding
eigenvalue $\lambda(h)$, satisfying $\lambda(h)\leq
h(b_0+\eta_1+\epsilon(h))$ for $h\in (0,h_0]$. We observe that
\[
H^h_{\tilde A}u_h=\lambda(h)u_h+r\,,
\]
where
\[
r=2(\tilde A-A)(hD_y+A)u_h+(hD_y(\tilde A-A)+(\tilde A-A)^2)u_h\,.
\]
By assumption, we get
\[
\|r\|\leq C(\|(hD_y+A)u_h\|_{L^2(\mathbb R^2\setminus
U_\eta)}+\|u_h\|_{L^2(\mathbb R^2\setminus U_\eta)})\,,
\]
where
\[
U_\eta=\{(x,y)\in \mathbb R^2 : b(x,y)<b_0+\eta\}\,.
\]
By Agmon estimates (see \cite{HM,HM01}), there exist $C$ and
$\gamma>0$ such that
\begin{equation}\label{e:Agmon}
\|(hD_y+A)u_h\|_{L^2(\mathbb R^2\setminus
U_\eta)}+\|u_h\|_{L^2(\mathbb R^2\setminus U_\eta)} \leq
Ce^{-\frac{\gamma}{h^{1/2}}}\,,
\end{equation}
that implies the estimate
\[
r=\mathcal O\left(e^{-\frac{\gamma}{h^{1/2}}}\right)\,.
\]

Now we can apply  \cite[Proposition 2.5]{HSI} (see also
\cite[Proposition 4.1.1]{He88}) and obtain the following estimates
for the non-symmetric distances between $E$ and $\tilde E$:
\[
\vec d(E,\tilde E)=\vec d(\tilde E,E)=\mathcal
O\left(e^{-\frac{\alpha}{h^{1/2}}}\right)\,,
\]
with some $\alpha>0$. As soon as we have proven these estimates, the
construction of $\Psi^h$ can be done essentially in the same way as
a similar construction in the proof of \cite[Theorem 2.4]{HSI} (see
also \cite[Theorem 4.2.1]{He88}).
\end{proof}

Theorem~\ref{t:comparison} allows us to continue our further
investigations under very strong conditions on $A$. For instance, we
can assume that $b$ is constant outside a compact set in $\mathbb
R^2$. In the sequel, we will assume that, for any $(k,\ell)\in
\mathbb N^2$ with $k+\ell>0\,$,
\[
\sup_{(x,y)\in \mathbb R^2}\left|\frac{\partial^{k+l}A}{\partial
x^k\partial y^\ell}(x,y)\right| <\infty\,.
\]
In particular, the functions $b$ and $\frac{\partial A}{\partial y}$
belong to the symbol class $S=S(1)\,$.

\section{Towards normal forms}\label{s:normalforms}
In this section, we will put the operator $H^h$ given by
\eqref{e:Hh} in a normal form, using very explicit transformations and
asymptotic expansions.

\subsection{Some transformations}\label{s:metaplectic}

First, we write the operator $H^h$ in the form
\[
H^h=-X^2_1-X^2_2\,,
\]
where
\[
X_1=h\partial_x, \quad X_2=h\partial_y+iA(x,y)\,.
\]

\subsubsection{Change of variables}
Now we make a change of variables
\begin{equation}\label{e:trans1}
x_1=A(x,y)\,,\quad y_1=y\,.
\end{equation}
In the new coordinates, we have
\[
{\bf B}=dx_1\wedge dy_1\,.
\]
Define functions $\hat b$ and $\hat A_y$ on $\mathbb R^2$ by
\[
\hat b(A(x,y),y)=b(x,y)\,, \quad  \hat A_y(A(x,y),y)=\frac{\partial
A}{\partial y} (x,y)\,,\quad  (x,y)\in \mathbb R^2\,.
\]
It is easy to see that $\hat b$ and $\hat A_y$ belong to the class $S$. By
\eqref{A00}, it follows that the minimum of $\hat b$ is at $(0,0)$
and
\begin{equation}\label{Ahat00}
\hat A_y(0,0)=0\,.
\end{equation}

For the operator
\[
\bar H^h(x_1,y_1,hD_{x_1},hD_{y_1})=H^h(x,y,hD_{x},hD_{y})\,,
\]
we obtain that
\[
\bar H^h(x_1,y_1,hD_{x_1},hD_{y_1})=-\bar X^2_1-\bar X^2_2\,,
\]
where
\begin{align*}
\bar X_1(x_1,y_1,hD_{x_1},hD_{y_1})= & \hat b(x_1,y_1)h\partial_{x_1}, \\
\bar X_2(x_1,y_1,hD_{x_1},hD_{y_1})= & \hat
A_y(x_1,y_1)h\partial_{x_1}+h\partial_{y_1}+ix_1.
\end{align*}
In the coordinates $(x_1,y_1)$, the flat Euclidean metric
$g=dx^2+dy^2$ is written as
\begin{multline*}
g= \hat b^{-2}(x_1,y_1) \,dx_1^2-2 \hat A_y(x_1,y_1)\hat
b(x_1,y_1)^{-2}\,dx_1\,dy_1\\ +(1+\hat b^{-2}(x_1,y_1)\hat
A_y(x_1,y_1)^2)\,dy_1^2\,.
\end{multline*}
The operator $\bar H^h$ is the magnetic Schr\"odinger operator
associated with this metric and the constant magnetic field. It is
self-adjoint with respect to the Riemannian volume form
\[
\sqrt{\det g}\,dx_1\,dy_1=\hat b(x_1,y_1)^{-1}\,dx_1\,dy_1\,.
\]

Now we move the operator $\bar H^h$ into the Hilbert space
$L^2(\RR^2, dx_1\,dy_1)$, using the unitary isomorphism
\begin{equation}\label{e:unitary}
u\in L^2(\RR^2, \hat b(x_1,y_1)^{-1}\,dx_1\,dy_1)\mapsto \hat
b(x_1,y_1)^{-1/2}u \in L^2(\RR^2,\,dx_1\,dy_1).
\end{equation}
For the corresponding operator $\hat{H}^h =\hat b^{-1/2}\bar H_h
\hat b^{1/2}$, we obtain that
\[
\hat{H}^h(x_1,y_1,hD_{x_1},hD_{y_1})=-\hat X^2_1-\hat X^2_2\,,
\]
where
\begin{align*}
\hat X_1(x_1,y_1,hD_{x_1},hD_{y_1}) = & \hat b^{-1/2}\bar
X_1(x_1,y_1,hD_{x_1},hD_{y_1}) \hat b^{1/2}\\ = & \hat
b(x_1,y_1)h\partial_{x_1}+\frac{1}{2}h\partial_{1}\hat b(x_1,y_1)
\end{align*}
and
\begin{align*}
\hat X_2(x_1,y_1,D_{x_1},D_{y_1})= & \hat b^{-1/2}\bar
X_2(x_1,y_1,D_{x_1},D_{y_1}) \hat b^{1/2}\\ = &
\hat A_y(x_1,y_1)h\partial_{x_1}+ h\partial_{y_1}+ix_1\\
& +\frac{1}{2}h[\hat A_y\hat b^{-1}
\partial_{1}\hat b+\hat b^{-1}\partial_2\hat b](x_1,y_1)\,.
\end{align*}
Here we use notation $\partial_{1}\hat b(x_1,y_1)=\frac{\partial
\hat b}{\partial x_{1}}(x_1,y_1)\,, \partial_{2}\hat
b(x_1,y_1)=\frac{\partial \hat b}{\partial y_{1}}(x_1,y_1)$\,.

\subsubsection{Metaplectic transformations}
Next, we make some metaplectic transformations.

{\bf Partial $h$-Fourier transform}\\
Using the partial Fourier transform in $y_1$ ($F : y_1\to y_2$), we
obtain
\[
Q^h(x_1,y_2,hD_{x_1},hD_{y_2})=
\hat{H}^h(x_1,-hD_{y_2},hD_{x_1},y_2) = -\tilde X^2_1-\tilde X^2_2\,,
\]
where
\begin{align*}
\tilde X_1(x_1,y_2,hD_{x_1},hD_{y_2})& =\hat
X_1(x_1,-hD_{y_2},hD_{x_1},y_2)\\ & = \hat
b(x_1,-hD_{y_2})h\partial_{x_1}+\frac{1}{2}h\partial_{1}\hat
b(x_1,-hD_{y_2})
\end{align*}
and
\begin{align*}
\tilde X_2(x_1,y_2,hD_{x_1},hD_{y_2}) =& \hat
X_1(x_1,-hD_{y_2},hD_{x_1},y_2)\\ = &
\hat A_y(x_1,-hD_{y_2})h\partial_{x_1}+i(y_2+x_1) \\
& + \frac{1}{2}h[\hat A_y\hat b^{-1}
\partial_{1}\hat b+\hat b^{-1}\partial_2\hat b](x_1,-hD_{y_2}).
\end{align*}

A further {\bf linear change of variables}
\begin{equation}\label{e:trans2}
x=x_1+{y_2}, y=-y_2
\end{equation}
gives for
\begin{equation}\label{hatT}
\hat T^h(x,{y},hD_{x},hD_{y};h)=Q^h(x+y,-y,hD_{x},hD_x-hD_y)
\end{equation}
the expression
\begin{equation} \label{hatTh0}
\hat T^h = -\check X^2_1-\check X^2_2\,,
\end{equation}
where
\begin{equation}
\begin{aligned}\label{checkX1}
\check X_1(x,{y},hD_{x},hD_{y};h)=& \tilde
X_1(x+y,-y,hD_{x},hD_x-hD_y)\\
= & \hat b(x+y,hD_{y} -hD_{x})h\partial_{x}
\\ & \quad +\frac{1}{2}h\partial_{1}\hat b(x+y,hD_{y} -hD_{x})
\end{aligned}
\end{equation}
and
\begin{equation}
\begin{aligned}\label{checkX2}
\check X_2(x,{y},hD_{x},hD_{y};h)=& \tilde
X_1(x+y,-y,hD_{x},hD_x-hD_y)\\  = &
\hat A_y(x+y,hD_{y}-hD_{x})h\partial_{x}+ix\\
& \quad + \frac{1}{2}h[\hat A_y\hat b^{-1}
\partial_{1}\hat b+\hat b^{-1}\partial_2\hat b](x+y,hD_{y}-hD_{x})\,.
\end{aligned}
\end{equation}

\subsubsection{Scaling} Finally, we make the {\bf dilation}
$x=h^\frac 12 \tilde x, y=\tilde y$. It should be noted that this
transformation is not metaplectic. Therefore, when we apply it we
leave the $h$-pseudodifferential calculus and loose the possibility
to use all the known results from this theory. Forgetting the tilde,
we get, after division by $h$, a more symmetric expression for the
operator
\[
\widetilde T^h(x,{y},D_{x},D_{y};h)=h^{-1}\widehat T^h(h^{\frac
12}x,y,h^{-\frac 12}D_{x},D_y;h)\,.
\]
We have
\begin{equation}\label{widetildeT}
\widetilde T^h = -\widetilde  X^2_1-\widetilde  X^2_2\,,
\end{equation}
where
\begin{multline}\label{widetildeX1}
\widetilde X_1(x,{y},D_{x},hD_{y};h)= h^{-\frac 12}\check
X_1(h^{\frac 12}x,y,h^{-\frac 12}D_{x},D_y;h)\\
= \hat b(h^{\frac 12}x+y,hD_{y}-h^{\frac
12}D_{x})\partial_{x}+\frac{1}{2}h^{1/2}\partial_{1}\hat b(h^{\frac
12}x+y,hD_{y}-h^{\frac 12}D_{x})
\end{multline}
and
\begin{multline}\label{widetildeX2}
\widetilde X_2(x,{y},D_{x},hD_{y};h)=h^{-\frac 12}\check
X_2(h^{\frac 12}x,y,h^{-\frac 12}D_{x},D_y;h)\\
=\hat A_y(h^{\frac 12}x+y,hD_{y}-h^{\frac 12}D_{x})\partial_{x}+ix
\\ +\frac{1}{2}h^{1/2}[\hat A_y\hat b^{-1}
\partial_{1}\hat b+\hat b^{-1}\partial_2\hat b](h^{\frac 12}x+y,hD_{y}-h^{\frac 12}D_{x})\,.
\end{multline}

The main problem is that the operator $\widetilde T_h$ is written as
a differential operator in $x$ and $D_x$ with pseudodifferential
coefficients in $h^{1/2}x+y, hD_y-h^\frac 12 D_{x}\,$. \\
In the next step, we will rewrite it as an $h$-pseudodifferential
operator in the $y$ variable with values in the class of
differential operators in the $x$ variable.

\subsection{Weyl calculus and justification of the
expansions}\label{s:justification}

For any $h>0$, the operators $h^{1/2}x+y$ and $hD_y-h^{1/2} D_x $
are commuting self-adjoint unbounded linear operators in
$L^2(\mathbb R^2, dx\,dy)$. Spectral theorem allows us to define the
operator $a(h^{1/2}x+y, hD_y-h^{1/2} D_x )$ as a bounded linear
operator in $L^2(\mathbb R^2, dx\,dy)$ for any $a\in S(1)$. In this
subsection, we derive an asymptotic expansion for the operator
$a(h^{1/2}x+y, hD_y-h^{1/2} D_x )$ in the form
\[
a(h^{1/2}x+y, hD_y-h^{1/2} D_x ) \sim \sum_{j\geq 0}h^\frac j2
\sum_{\ell_j=1}^{j+1} b_{j,\ell_j}(x,D_x) a_{j,\ell_j}(y, h D_y)
\]
with some $b_{j,\ell_j}\in S^*({\mathbb R}^2)$ and $a_{j,\ell_j}\in
S$. Here $b(x,D_x)$ denotes the Weyl quantization of the symbol
$b\in S^*({\mathbb R}^2)$ and $a(y, h D_y)$ is the semiclassical
pseudodifferential operator with Weyl symbol $a\in S(1)\,$.

First, we consider the case when $a\in \mathcal S(\mathbb R^2)$\,.
Then we write
\[
a(h^{1/2}x+y, hD_y-h^{1/2} D_x ) = \int \hat a (\tau_1,\tau_2) e^{i
[\tau_1(h^{1/2} x+y) + \tau_2 (hD_y -h^{1/2} D_x)]}
 d\tau_1\,d\tau_2\,,
\]
where $\hat a(\tau_1,\tau_2)$ is the Fourier transform of $a$. \\
This formula can be rewritten in the form
\begin{multline}\label{e:a}
a( h^{1/2}x+y, hD_y-h^{1/2} D_x ) \\ = \int \hat a (\tau_1,\tau_2)
e^{i [\tau_1 y + \tau_2 (hD_y )] }\, e^{ih^{1/2} (\tau_1 x - \tau_2
D_x)}
 d\tau_1\,d\tau_2\,.
\end{multline}
Observe that
\begin{equation}\label{e:Weyl}
\int \hat
a(\tau_1,\tau_2)e^{i(\tau_1y+\tau_2hD_y)}\,d\tau_1\,d\tau_2=a(y,hD_y)\,.
\end{equation}
Indeed, we have
\[
e^{i(\tau_1y+\tau_2hD_y)}u(y)=e^{\frac12ih\tau_1\tau_2}e^{i\tau_1y}u(y+h\tau_2)\,.
\]
Using this formula, for the operator
\[
A=\int \hat
a(\tau_1,\tau_2)e^{i(\tau_1y+\tau_2hD_y)}\,d\tau_1\,d\tau_2\,,
\]
we get
\begin{align*}
Au(y)=& \int \hat
a(\tau_1,\tau_2)e^{\frac12ih\tau_1\tau_2}e^{i\tau_1y}u(y+h\tau_2)\,d\tau_1\,d\tau_2\\
=& \frac{1}{(2\pi)^2}\int a(\xi_1,\xi_2)
e^{-i(\tau_1\xi_1+\tau_2\xi_2)} e^{\frac12ih\tau_1\tau_2}
e^{i\tau_1y}
u(y+h\tau_2)\,d\tau_1\,d\tau_2\,d\xi_1\,d\xi_2\\
=& \frac{1}{(2\pi)^2} \int
d\tau_1\,e^{-i\tau_1(\xi_1-\frac12h\tau_2-y)} \int e^{-i\tau_2\xi_2}
a(\xi_1,\xi_2) u(y+h\tau_2) \,d\tau_2\,d\xi_1\,d\xi_2.
\end{align*}
Now we use the Fourier transform inversion formula in $\tau_1$ and
$\xi_1$:
\[
Au(y) = \frac{1}{2\pi} \int e^{-i\tau_2\xi_2}
a\left(\frac12h\tau_2+y,\xi_2\right) u(y+h\tau_2) \,d\tau_2\,d\xi_2\,.
\]
Finally, we make the change of variables $x=y+h\tau_2$ and get
\[
Au(y) = \frac{1}{2\pi h} \int e^{\frac{i}{h}(y-x)\xi_2}
a\left(\frac{x+y}{2},\xi_2\right) u(x) \,dx\,d\xi_2 =a(y,hD_y)\,,
\]
that completes the proof of \eqref{e:Weyl}.

We can then expand the right hand side of \eqref{e:a} in powers of
$h^{1/2}$ and get
\begin{multline}\label{e:expand-e:a}
a(h^{1/2}x+y, hD_y-h^{1/2} D_x ) \\ =\sum_{k=0}^N
\frac{1}{k!}i^kh^{k/2} \int (\tau_1 x - \tau_2 D_x)^k \hat a
(\tau_1,\tau_2)  e^{i [\tau_1 y + \tau_2 (hD_y )]
}\,d\tau_1\,d\tau_2 +R_N(h) \,,
\end{multline}
where
\begin{multline*}
R_N(h)=\frac{1}{N!}i^{N+1}h^{(N+1)/2} \int \hat a (\tau_1,\tau_2)
e^{i [\tau_1 y + \tau_2 (hD_y )] }\times \\
\times \left(\int_0^1(1-t)^N e^{ith(\tau_1 x - \tau_2 D_x)} (\tau_1
x - \tau_2 D_x)^{N+1}dt \right) \,d\tau_1\,d\tau_2.
\end{multline*}
For $k\in \mathbb N$, consider the Sobolev space $B^k(\mathbb R)$
given by
\begin{equation}\label{e:defBk}
B^k(\mathbb R)=\{u\in L^2(\mathbb R) : x^\alpha D^\beta_xu \in
L^2(\mathbb R)\ {\rm for}\ \alpha+\beta \leq k\}\,.
\end{equation}
Then for any $h\in (0,1)$ and for any $s\geq 0$, there exists
$C_s>0$ such that, for any $h\in (0,h_0]$, we have
\[
\|R_N(h) : B^{s+N+1}(\mathbb R_x)\times L^2(\mathbb R_y)\to
B^{s}(\mathbb R_x)\times L^2(\mathbb R_y) \| \leq C_s
h^\frac{N+1}{2}\,.
\]

Using \eqref{e:Weyl}, we can compute explicitly the first
coefficients in the expansion~\eqref{e:expand-e:a}. For instance,
for $N=2$, we get
\begin{multline}\label{e:a-explicit}
a(h^{1/2}x+y, hD_y-h^{1/2} D_x )\\
= a(y,hD_y)+ ih^{1/2}(x (D_1a)(y,hD_y)-
D_x (D_2a)(y,hD_y))\\
- \frac{1}{2}h (x^2 (D^2_1a)(y,hD_y)-(xD_x+D_xx) (D_1D_2a)(y,hD_y)
+ D_x^2 (D^2_2a)(y,hD_y)) \\
+ R_2(h)\,.
\end{multline}
This is what we would have obtained by considering the non
commutative Taylor expansion with respect to $h^{1/2} x$ and
$h^{1/2} D_x$ at the "point" $(y, hD_y)\,$. More generally, we get
the following result.

\begin{prop}\label{propA1}
If $a$ is a semiclassical symbol in $\mathbb R^2$, then we have:
\begin{equation}\label{expa}
a(h^{1/2}x+y, hD_y-h^{1/2} D_x,h ) \sim \sum_{j\geq 0} h^{\frac j2}
\left(\sum_{k_1 + k_2 \leq j} a_{k_1,k_2,j}(y,hD_y) x^{k_1}
D_x^{k_2}\right)\,, \end{equation} where $a_{000}(y,\eta)=
a(y,\eta)$, and the remainder $R_N(h)$ defined for any $N\in \mathbb
N$ by
\begin{multline*}
R_N(h)=a(h^{1/2}x+y, hD_y-h^{1/2} D_x,h )\\ - \sum_{j=0} ^N h^{\frac
j2} \left(\sum_{k_1 + k_2 \leq j} a_{k_1,k_2,j}(y,hD_y) x^{k_1}
D_x^{k_2}\right)\,,
\end{multline*}
satisfies the following condition: for any $N\in \mathbb N$ there
exists $h_0>0$ such that, for any $h\in (0,h_0]$ and for any $s\geq
0$, we have
\[
\|R_N(h) : B^{s+N+1}(\mathbb R_x)\times L^2(\mathbb R_y)\to
B^{s}(\mathbb R_x)\times L^2(\mathbb R_y) \| \leq C_s
h^\frac{N+1}{2}\,.
\]
\end{prop}

\begin{rem}
The (standard) problem is that $a$ is not defined everywhere. But if
one has some information on the frequency set of the quasi-mode, one
can assume that $a$ is compactly supported (or has an extension to a
semiclassical symbol in $S$). The results then will not depend on
the choice of the extension.
\end{rem}

\begin{rem}
On the right hand side of \eqref{expa}, the operators will be
applied on expression of the form
$$
w(x,y,h)\sim \sum_\ell h^{\frac \ell 2} \left(\sum_{k=0}^{k_\ell}
u_{\ell,k} (y,h) v_{\ell, k} (x)\right) \,,
$$
where the $u_{\ell,k} (y,h)$ have their $h$-microsupport close to
$(0,0)$ (our symbols in $(y,\eta)$ are only defined there) and the
$v_{\ell,k}$ are functions in $\mathcal S(\mathbb R)$ (actually
Gaussians multiplied by polynomials).
\end{rem}

\begin{rem}
For the treatment of  the action of $\widetilde T^h$, we have to
compose the expansions obtained in Proposition \ref{propA1} with
$x^kD_x^\ell$ and sum various terms of this type.
\end{rem}

\begin{rem}
Together with the maximal estimates on the eigenfunctions (see the
estimates \eqref{control1} below), we have the possibility to stop
the expansion in degree $N$, the remainder being controlled. One can
then follow formal Grushin's method using finite expansions in
powers of $h$ and give a non formal meaning to all the constructions
modulo an error term of order $\mathcal O (h^{\tilde N})$
where $\tilde N$ depends on $N$ and can be made arbitrarily large by
choosing $N$ large enough.
\end{rem}

\subsection{The explicit expansion}\label{s:dilation}
In this subsection, we will use the results of
Subsection~\ref{s:justification} to rewrite the operator $\widetilde
T^h$ as an $h$-pseudodifferential operator in the $y$ variable with
values in the class of differential operators in the $x$ variable:
\begin{equation}\label{hatTh1}
\widetilde T^h(x,{y},D_{x},D_{{y}};h)\sim \sum_{k=0}^\infty
h^{k/2}S_k(x,{y},D_{x},hD_{{y}}),
\end{equation}
For this purpose, we first expand the coefficients in $h^{1/2}x$ and
$h^{1/2} D_x$ in the formulae \eqref{widetildeX1} and
\eqref{widetildeX2}. By \eqref{e:a-explicit}, we obtain that
\begin{multline*}
\frac{1}{i}\widetilde X_1 = \hat b(y,hD_y)D_x \\
\begin{aligned}
 & +
ih^{1/2}(\frac12(xD_x+D_xx) (D_1\hat b)(y,hD_y)-
D^2_x (D_2\hat b)(y,hD_y))\\
& - h (\frac14 (x^2D_x+D_xx^2) (D^2_1\hat b)(y,hD_y) -
\frac12(xD^2_x+D^2_xx)(D_1D_2\hat b)(y,hD_y)\\ & + \frac12 D_x^3
(D^2_2\hat b)(y,hD_y))+ {\mathcal O}(h^{3/2})
\end{aligned}
\end{multline*}
and
\begin{align*}
\frac{1}{i}\widetilde X_2 = & \hat A_y(y,hD_y)D_x+x+
h^{1/2}[(xD_x+D_xx)(\partial_1\hat A_y)(y,hD_y)\\ & - D^2_x
(\partial_2\hat A_y)(y,hD_y) +\frac{1}{2}(-D_1(\hat A_y\hat
b^{-1})\hat b+\hat b^{-1}D_2\hat b)(y,hD_y)]\\ & + h
[-\frac14(x^2D_x+D_xx^2) (D^2_1\hat A_y)(y,hD_y)\\
& +\frac12(xD^2_x+D^2_xx)(D_1D_2\hat A_y)(y,hD_y)
-\frac12 D_x^3 (D^2_2\hat A_y)(y,hD_y)\\
& + \frac12i (x (-D_1[D_1(\hat A_y\hat b^{-1})\hat b]+D_1[\hat
b^{-1}D_2\hat b])(y,hD_y)\\ & - D_x (-D_2[D_1(\hat A_y\hat
b^{-1})\hat b]+D_2[\hat b^{-1}D_2\hat b])(y,hD_y))]+ {\mathcal
O}(h^{3/2}).
\end{align*}
Next, we substitute these asymptotic formulas into
\eqref{widetildeT} that gives the desired asymptotic expansion
\eqref{hatTh1}.

We now compute the first two coefficients in this expansion.

For the coefficient $S_0$, we get:
\begin{multline}\label{S0}
S_0(x,{y},D_{x},hD_{{y}}) \\ = (\hat b^2+\hat
A_y^2)({y},hD_{y})D_{x}^2 - \hat
A_y({y},hD_{y})(x\,D_{x}+D_{x}x)+x^2\,.
\end{multline}
The Weyl vector valued $h$-symbol of $S_0$ is given by
\begin{align}\label{sigma0}
\sigma_0(x,D_x, y,\eta)= (\hat b^2(y,\eta)+\hat
A_y^2(y,\eta))D_{x}^2 -\hat A_y(y,\eta)(xD_{x}+D_{x}\,x)+x^2\,.
\end{align}
For a fixed $(y,\eta)$, this is an harmonic oscillator, whose
spectrum is given by
\begin{equation}\label{LL}
\lambda_k(y,\eta)=(2k+1)\hat b(y,\eta)\,, \quad k\in \mathbb N\,.
\end{equation}
This could seem surprising but one way to recognize this simply is
to observe that, by a gauge transformation $\exp \left(i
\frac{\alpha}{(b^2 + \alpha^2)} \frac{x^2}{2}\right) $, with $\alpha
=\hat A_y(y,\eta)$ and $b=\hat b(y,\eta)$,
$\sigma_0(x,D_{x},y,\eta)$ is unitary equivalent to $(b^2 +
\alpha^2) D_x^2 + \frac{b^2}{b^2 +\alpha^2} x^2$. An additional
dilation permits us to arrive at $b(D_x^2 + x^2)\,$. In particular,
we get for the $L^2$-normalized ground state of $\sigma_0(x,D_x,
y,\eta)$:
\begin{equation}\label{defhyeta}
h_{y,\eta}(x) = \rho (y,\eta) \exp - \delta(y,\eta) x^2\,,
\end{equation}
with $\rho(y,\eta) >  0$ and $\Re \delta (y,\eta) >0$, $\rho$ and
$\delta$ depending smoothly on $(y,\eta)$.

The coefficient $S_1$ is given by
\begin{align*}
S_1= & i\Big(\frac12(xD^2_x+2D_xxD_x+D^2_xx) (\hat bD_1\hat
b)(y,hD_y)- 2D^3_x (\hat bD_2\hat b)(y,hD_y)\Big)\\ & + (\hat
A_y(y,hD_y)D_x+x)
\Big[(xD_x+D_xx)(\partial_1\hat A_y)(y,hD_y) \\
& - D^2_x (\partial_2\hat A_y)(y,hD_y) +\frac{1}{2}(-D_1(\hat A_y\hat b^{-1})\hat b+\hat b^{-1}D_2\hat b)(y,hD_y)\Big]\\
& + \Big[(xD_x+D_xx)(\partial_1\hat A_y)(y,hD_y)- D^2_x (\partial_2\hat A_y)(y,hD_y)\\
& +\frac{1}{2}(-D_1(\hat A_y\hat b^{-1})\hat b+\hat b^{-1}D_2\hat
b)(y,hD_y)\Big](\hat A_y(y,hD_y)D_x+x)\,.
\end{align*}
We observe that $S_1$ inverses the parity in the $x$ variable.

Computation of $S_2$ is rather lengthy and we will omit it here. We
only observe that $S_2$ respects the parity in the $x$ variable.

These computations could be useful  for determining the
sub-principal symbol of the effective operator $p_{\rm
eff}(y,hD_y;h,z)$.  We will explain this in Section \ref{s7}.  But
for proving the existence of the symbol $p_{\rm eff}(y,\eta;h,z)$,
we need only the structure of the operators $S_j$.

\section{Eigenfunctions estimates}\label{s:eigenfunctions}
\subsection{On the frequency set of eigenfunctions}\label{s:frequency} The frequency set
was introduced by V. Guillemin and S. Sternberg \cite{GuiSt} but we
prefer for our need to refer to the books of D. Robert \cite{Rob} or
M. Zworski \cite{Z}. This is the analog of the wave front set of
H\"ormander in the semi-classical context.

\begin{defn} Given an open subset $\Omega $ of ${\mathbb R}^m$ and a map $h\in
(0,h_0]\mapsto T_h\in {\mathcal D}^\prime (\Omega)\,$, a point
$(x_0,p_0)\in {\mathbb R}^m_x\times {\mathbb R}^m_p$ is not in the
frequency set $F[T_h]$ of $T_h$ if there exists $\phi\in
C^\infty_c({\mathbb R}^m)$ such that $\phi(x_0)\neq 0$ and a
neighborhood $V_{p_0}$ of $p_0$ such that
\[
<\phi(x)e^{-ih^{-1}x\cdot p}, T_h>=\mathcal O(h^\infty)\,,\quad
h\rightarrow 0\,,
\]
uniformly with respect to $p\in V_{p_0}\,$.\end{defn}

There exists also an $h$-pseudodifferential characterization of the
frequency set for a family $T_h$ in $L^2$\,. A point $(x_0,p_0)\in
{\mathbb R}^m_x\times {\mathbb R}^m_p$ is not in the frequency set
$F[T_h]$ of $T_h$ if there exists an $h$-pseudodifferential operator
$\chi (x, hD_x)$ whose symbol is elliptic at $(x_0,p_0)$ such that
$\chi (x, hD_x) T_h = \mathcal O (h^\infty)\,$ in $L^2$.

The following result is rather standard:

\begin{prop}
Suppose that $u_h$ is an $L^2$ normalized eigenfunction of $H_h$
corresponding to an eigenvalue $\lambda_h$ such that $\lambda_h \leq
C h$ for $h\in (0,1)$. Then the frequency set of $u_h$ is  non empty
and contained in
\[
F[u_h]\subset \{(x,y,\xi,\eta)\in T^*{\mathbb R}^2 : \xi=0\,,
\eta=-A(x,y), b(x,y)\leq C\}\,.
\]
\end{prop}
\begin{proof}
This is just a combination of the elliptic theory for
$h$-pseudo-differen\-ti\-al operators combined with Agmon estimates
(see above around  \eqref{e:Agmon}).
\end{proof}

\begin{rem}
As a consequence, if we consider a cut-off function $\chi$ equal to
$1$ on a fixed neighborhood of $b^{-1} ((-\infty,C))$ and with
support in $b^{-1} ((-\infty, C'))$ with some $C^\prime >C$, then
$\chi u_h$ has the same frequency set and satisfies:
$$
(H^h-\lambda_h) (\chi u_h) = \mathcal O (h^\infty)\,.
$$
\end{rem}

We can now follow the frequency set by change of coordinates or more
generally by the action of $h$-Fourier integral operators (this
includes the $h$-Fourier transform) the transformation being given
by the associated canonical transformation. Let us consider
transformations introduced in Subsection~\ref{s:metaplectic}.

After the change of variables \eqref{e:trans1}, we get for the
transformed eigenfunction $\hat u_h(x_1,y_1)=u_h(x,y)$,
\[
F[\hat u_h]\subset \{(x_1,y_1,\xi_1,\eta_1)\in T^*{\mathbb R}^2 :
\xi_1=0, \eta_1=-x_1\,, \,\hat b(x_1,y_1)\leq C\}\,.
\]
Now we apply the unitary isomorphism \eqref{e:unitary}. For the
transformed eigenfunction
\[
v_h(x_1,y_1)=\hat b(x_1,y_1)^{-1/2} u_h(x_1,y_1)\,,
\]
we obtain that
\[
F[v_h]=F[u_h]\subset \{(x_1,y_1\,,\xi_1,\eta_1)\in T^*{\mathbb R}^2
: \xi_1=0,\, \eta_1=-x_1, \,\hat b(x_1,y_1)\leq C\}\,.
\]
Next we make the partial Fourier transform in $y_1\,$. So the
corresponding eigenfunction $w_h$ is the partial Fourier transform
in $y_1$ of $v_h$, and therefore
\begin{align*}
F[w_h]= & \{(x_1,y_2,\xi_1,\eta_2)\in T^*{\mathbb R}^2 :
(x_1,-\eta_2,\xi_1,y_2)\in F[v_h]\}\\ \subset &
\{(x_1,y_2,\xi_1,\eta_2)\in T^*{\mathbb R}^2 : \xi_1=0, y_2=-x_1\,,
\hat b(x_1,-\eta_2)\leq C\}\,.
\end{align*}
Next we make a change of variables  \eqref{e:trans2}, which gives
for the corresponding eigenfunction $\hat{u}_h$
\begin{equation}\label{e:freq-set}
F[\hat{u}_h]\subset \{(x,y,\xi,\eta)\in T^*{\mathbb R}^2 : x=0\,,\,
\xi=0\,, \,\hat b(y,\eta)\leq C\}\,.
\end{equation}

Finally, we make the dilation $x=h^{1/2} \tilde x\,, \, y=\tilde
y\,$. The corresponding eigenfunction is given by $\widetilde
u_h(\tilde x, \tilde y)=h^{-1/2}\hat{u}_h(h^{-1/2}\tilde x, \tilde
y)$. Note that, in this step, we do not control the frequency set,
but, for any natural $k$ and $\ell$, the asymptotic behavior of the
norm of $x^k D_x^\ell \widetilde u_h$ as $h\to 0$ is well controlled
as will be shown in Subsection \ref{AppH}.

\subsection{Maximal estimates}\label{AppH}
 In the following, we will use the
asymptotic expansions of Subsection~\ref{s:justification} applied to
the eigenfunction $\widetilde u_h$ of the operator $\widetilde T_h$
introduced above. To analyze the action of the remainder in these
asymptotic expansions on $\widetilde u_h$, we have to control the
$L^2$ norm of $x^\alpha D_x^\beta \widetilde u_h$. Formally the
possibility of such a control seems reasonable taking into account
the information that $\widetilde T_h^k \widetilde u_h = \lambda_h^k
\widetilde u_h = \mathcal O (h^k)$ in $L^2\,$. To prove the
corresponding estimate rigorously, we observe that before the
metaplectic transformations introduced in
Subsection~\ref{s:metaplectic} such an estimate  is related to a
``regularity'' estimate (\`a la H\"ormander) for polynomials of
vector fields  or more precisely (\`a la Helffer-Nourrigat) for the
iterates of the magnetic Laplacian.

More precisely, in the flat Euclidean space ${\mathbb R}^3$ with
coordinates $(x,y,t)\,$, consider the vector fields
\[
Y_1 = \frac{\partial}{\partial x}, \quad Y_2 =
\frac{\partial}{\partial y} + A(x,y) \frac{\partial}{\partial t}\,.
\]
Then we have
\[
[Y_1,Y_2](x,y,t)= b(x,y)\frac{\partial}{\partial t}\neq 0\,.
\]
Thus, for any $(x,y,t)\in {\mathbb R}^3$, the vector fields $Y_1,
Y_2$ satisfy the H\"ormander condition $(C.H)_{2,(x,y,t)}$
\cite[Chapter I, \S 1]{HN}, which means that the vectors
$Y_1(x,y,t), Y_2(x,y,t)$ and $[Y_1,Y_2](x,y,t)$ span the tangent
space $T_{(x,y,t)}{\mathbb R}^3\,$.

Consider the operator
\[
P=Y_1^2+Y_2^2\,.
\]
By Theorem 1.3 from \cite[Chapter IX]{HN}, we get, for any $u\in
C^\infty_c(V\times \mathbb R)$, where $V$ is a sufficiently small
neighborhood of $(0,0)\,$, that there exists $C$ such that
\[
\sum_{\alpha_1+\alpha_2\leq 2} \|Y_1^{\alpha_1}Y_2^{\alpha_2} u \|^2
\leq C_\alpha ( \|P u\|^2 + || u||^2) \,.
 \]
Similarly, for any $N$ there exists $C_{N}$ such that
\[
\sum_{\alpha_1+\alpha_2\leq 2 N}\|Y_1^{\alpha_1}Y_2^{\alpha_2} u
\|^2 \leq C_{N} (\|P^N u \|^2 + \|u\|^2)\,.
\]
Taking the partial Fourier transform in the $t$-variable, we get
\begin{multline*}
\sum_{\alpha_1+\alpha_2\leq 2 N}\|\left(\frac{\partial}{\partial
x}\right)^{\alpha_1}\left(\frac{\partial}{\partial y} + iA(x,y)
\tau\right)^{\alpha_2} u \|^2\\
\leq C_{N} (\|\left(\left(\frac{\partial}{\partial
x}\right)^2+\left(\frac{\partial}{\partial y} + iA(x,y)
\tau\right)^2\right)^N u \|^2 + \|u\|^2)\,.
\end{multline*}
Dividing by $\tau^{4N}$ and introducing $h=\tau^{-1}$, we obtain that
\[
\sum_{|\alpha|\leq 2 N}  h^{4N-2(\alpha_1+\alpha_2)} \|
X_1^{\alpha_1}X_2^{\alpha_2} u \|^2 \leq C_{N}\left( \|(H^h)^N u
\|^2 + h^{4N} \|u\|^2\right)\,.
\]
\begin{rem} Alternately, we could have used the Boutet de Monvel
results on hypoelliptic operators with multiple characteristics (see
in \cite{BGH}) in the symplectic case.
\end{rem}

Hence we get
\begin{prop}
If $u_h$ is an $L^2$ normalized eigenfunction of $H^h$ corresponding
to an eigenvalue $\lambda_h$ such that $\lambda_h \leq C h$ for
$h\in (0,1)$, then for any $\alpha=(\alpha_1,\alpha_2)\in \mathbb
N^2$ there exists a constant $C_{\alpha}$ such that, for $h\in
(0,1)$,
\begin{equation}\label{dec}
 \| X_1^{\alpha_1}X_2^{\alpha_2} u_h
\|_{L^2(\mathbb R^2)} \leq C_{\alpha} h^{(\alpha_1+\alpha_2)/2}\,.
\end{equation}
\end{prop}
The main contribution to the norm in \eqref{dec} of course comes
from the set $\{(x,y)\in \mathbb R^2 : b(x,y) \leq C\}$. Outside
this set, we have an exponential decay due to Agmon estimates (see
\cite{HM}).
\begin{lemma}
With the same assumptions, if $K$ is a compact such that $K \cap \{(x,y)\in \mathbb R^2 :
b(x,y) \leq C\} =\emptyset\,$, then there exist
$\epsilon=\epsilon_K>0$  and, for any $\alpha=(\alpha_1,\alpha_2)\in
\mathbb N^2$, $C_{K,\alpha}$  and $h_{K,\alpha}$ such that, for
$h\in (0,h_{K,\alpha})$,
$$
\| X_1^{\alpha_1}X_2^{\alpha_2} u_h \|_{L^2(K)}^2 \leq C_{K,\alpha}
e^{- \frac{\epsilon}{h^{1/2}}}\,.
 $$
\end{lemma}

We now follow  \eqref{dec} in the chain of transformations leading
to our normal form (see Subsection~\ref{s:metaplectic}). For the
transformed eigenfunction $\hat u_h$, we obtain that for any
$\alpha=(\alpha_1,\alpha_2)\in \mathbb N^2$ there exists a constant
$C_{\alpha}$ such that, for $h\in (0,1)$
\begin{equation}\label{dec0}
 \|\check X_1^{\alpha_1}\check X_2^{\alpha_2} \hat u_h
\|_{L^2(\mathbb R^2)} \leq C_{\alpha} h^{(\alpha_1+\alpha_2)/2}\,.
\end{equation}

By \eqref{checkX1} and \eqref{checkX2}, it follows that
\[
x= q_{11} \check X_1 + q_{12}  \check X_2 + h q_1\,,\,\quad hD_x =
q_{21} \check X_1 + q_{22}  \check X_2 + h q_2\,,
\]
where $q_{ij}$ and $q_i$ are $h$-pseudodifferential operators. Using
this decomposition, we get from \eqref{dec0} that for any
$k$ and $\ell$ in $\mathbb N$
\[
x^\ell (hD_x)^m \widehat u_h =  \mathcal O (h^{(m+\ell)/2})\,,
\]
in $L^2(\mathbb R^2)$.\\
This is what is needed for controlling the expansions which were
done after the scaling $x=h^{1/2} \tilde x, y=\tilde y$. With this
scaling we get that for any  $k$ and $\ell$ in $\mathbb N$
\begin{equation}\label{control1}
x^\ell \,D_x^m \widetilde u_h = \mathcal O (1)\,.
\end{equation}

\section{The Grushin method}\label{s:Grushin}
In this section, we construct an appropriate Grushin problem
in a neighborhood of the minimum point $(0,0)$ and apply Grushin's
method.

\subsection{Classes of pseudo-differential operators}\label{s:PDO} First, let us
recall a specific class of pseudo-differential operators which
appear to be useful in the analysis of fine spectral properties of
globally elliptic operators. We refer to Helffer \cite{He84} or
Shubin \cite{Sh} for this specific class which is of course
contained in the general class considered  in the Weyl calculus by
H\"ormander \cite{Hor}. The class $S^{m}(\mathbb R^2)$ is defined
as the set of $C^\infty$ function on $\mathbb R^2$ such that for any
natural $k,\ell$, there exists a constant $C_{k,\ell}$ such that
  \[
  |D_x ^kD_\xi^\ell a(x,\xi)|\leq C_{k\ell}(1+|x|+|\xi|)^{m-k-\ell}\,.
  \]
We associate with a symbol $a\in S^{m}(\mathbb R^2)$ an operator via
the Weyl quantization. We denote by $\mbox {Op}\, S^m$ the
corresponding class of operators which are well defined on $\mathcal
S(\mathbb R)$ and $\mathcal S'(\mathbb R)$. When $m=0$, these
operators are continuous in $L^2(\mathbb R)$.  As usual there is a
natural notion of principal symbol and of globally elliptic symbol.
For $m>0$, we say that the symbol $a\in S^m(\mathbb R^2)$ is
elliptic if there exists a constant $C>0$ such that
  $$
  a(x,\xi) \geq \frac 1C ((1+|x|+|\xi|)^{m} - C\,.
  $$

Globally elliptic  operators have parametrices, this means that
there exists a pseudo-differential operator $Q=\operatorname{ Op}
(q)$ in $\operatorname{ Op}\,S^{-m}$ with principal symbol equal to
$\frac{1}{a}$ for $|x|+|\xi|$ large enough such that
  $$
  \operatorname{ Op } (q)  \circ \operatorname{ Op } (a) = I + \mathcal R\,,
  $$
where $\mathcal R$ is regularizing in the sense that it has a
distribution kernel in $\mathcal S(\mathbb R\times \mathbb R)$
(equivalently that  it has a Weyl symbol in $\mathcal S(\mathbb
R^2)$ or that it  can be extended as a map from $\mathcal S'$ into
$\mathcal S$).\\ In addition, if we know by other means that $Q$ is
invertible then the inverse is itself a pseudo-differential operator
in $\operatorname{Op} S^{-m}$  (this is a special (easier) case of
the so-called Beals theorem).

It is also natural to introduce  a class of symbols $S^{0,m}
(\mathbb R^2\times \mathbb R^2)$ of the form
$$
(x,\xi,y,\eta)\in \mathbb R^2\times \mathbb R^2 \mapsto
b(x,\xi,y,\eta)\in \mathbb C,
$$
verifying the following estimates:
\[
|D_{y,\eta}^\alpha D_{x,\xi}^\beta b (x,\xi,y,\eta)| \leq
C_{\alpha,\beta} (1+|x|+|\xi|)^{m-|\beta|}\,, \, \forall (x,\xi)\in
\mathbb R^2\,,\, \forall (y,\eta)\in \mathbb R^2\,.
\]
These symbols could also depend on an additional parameter $h$ and
can be possibly expanded in powers of $h$ (with fixed $m$).\\
With an arbitrary symbol $b\in S^{0,m} (\mathbb R^2\times \mathbb
R^2)$, we can associate by the Weyl quantization a global
pseudodifferential operator $b(x,D_x,y,hD_y)$, which is
semi-classical in the $y$ variable. This operator acts on $\mathcal
S(\mathbb R_x) \widehat \otimes C_0^\infty(\mathbb R_y)$ by the
formula
\begin{multline*}
b(x,D_x,y,hD_y)w(x,y,h):=\\
h^{-1} \int b(\frac{x+x'}{2},\xi, \frac{y+y'}{2},\eta) w(x',y') e^{i
(x-x')\cdot \xi + i \frac{(y-y')\cdot \eta}{h}}\,dy' d\eta dx' d\xi
\,.
\end{multline*}
The class of such operators will be denoted by $\operatorname{Op}
S^{0,m}$.

One can consider an operator in the  class $\operatorname{Op}
S^{0,m}$ as an $h$-pseu\-do\-diffe\-ren\-ti\-al operator on
$C^\infty_0(\mathbb R_y)$ with a vector-valued symbol, taking values
in the space of global pseudodifferential operators on $\mathcal S
(\mathbb R_x)$. The Weyl vector valued $h$-symbol of the operator
$b(x,D_x,y,hD_y)$ is given by
\[
b(y,\eta)w(x)=b(x,D_x,y,\eta)w(x):= \int
b(\frac{x+x'}{2},\xi,y,\eta) w(x') e^{i (x-x')\cdot \xi}\,dx' d\xi
\,.
\]

For two operators $b(x,D_x,y,hD_y)\in \operatorname{Op} S^{0,*}$ and
$c(x,D_x,y,hD_y)\in \operatorname{Op} S^{0,*}$ we will denote by
$b(x,D_x,y,hD_y)\circ c(x,D_x,y,hD_y)$ their composition as
operators on $\mathcal S (\mathbb R_x) \hat\otimes
C^\infty_0(\mathbb R_{y})$ and by $b(y,\eta) c(y,\eta)$ the
(pointwise at $(y,\eta)$) composition of their Weyl vector valued
$h$-symbols $b(y,\eta)$ and $c(y,\eta)$ as global pseudodifferential
operators on $\mathcal S (\mathbb R_x)$.

We introduce the class $\operatorname{Op} S^{0,m}[h]$, which
consists of families $\{C(h) : h>0\}$ of bounded operators on
$\mathcal S (\mathbb R_x) \hat\otimes L^2(\mathbb R_y)$, which can
be represented as an asymptotic sum of the following type:
\[
C(h)\sim \sum_{j\geq 0} h^{\frac j2}c_j(x,D_x,y,hD_y)
\]
where each $c_j$ belongs to $S^{0,m_j} (\mathbb R^2\times\mathbb
R^2)\,$ with some $m_j$ and can be represented as a finite sum
\[
c_j(x,D_x,y,hD_y)=\sum b^{(j)}_\ell (y,hD_y) a^{(j)}_\ell (x,D_x)
\]
with some $b^{(j)}_\ell\in S(1)$ and $a^{(j)}_\ell\in
S^{m_j}(\mathbb R)$.

The asymptotic sum means that for any $N\in \mathbb N$ the remainder
\[
R_N(h)=C(h)- \sum_{j=0}^N h^{\frac j2}   c_j^w (x,D_x,y,hD_y)
\]
has the property that there exists $k(N)\in \mathbb N$ and $h_0>0$
such that, for any $h\in (0,h_0]$ and for any $s>0$, we have
\begin{equation}\label{e:RNh}
\|R_N(h) : B^{s+k(N)}(\mathbb R_x)\hat \otimes L^2(\mathbb R_y)\to
B^{s}(\mathbb R_x)\hat\otimes L^2(\mathbb R_y) \| \leq C_s
h^\frac{N+1}{2}\,.
\end{equation}

We also consider formal pseudodifferential operators of class
$\operatorname{Op} S^{0,*}[h]$, which are formal sums of the
following type:
\[
C(h)= \sum_{j\geq 0} h^{\frac j2}c_j(x,D_x,y,hD_y)\,,
\]
where each $c_j$ belongs to $S^{0,m_j} (\mathbb R^2\times\mathbb
R^2)\,$ with some $m_j$.

By Proposition~\ref{propA1}, it follows that, for a semiclassical
symbol $a$ on $\mathbb R^2$, the operator $a(h^{1/2}x+y,
hD_y-h^{1/2} D_x,h)$ belongs to $\operatorname{Op} S^{0,*}[h]$. By
\eqref{widetildeT}, this implies that the operator $\widetilde T^h$
belongs to $\operatorname{Op} S^{0,2}[h]$.

\begin{defn}
Let $\Omega$ be an open subset of ${\mathbb R}^2$. For an operator
$C\in \operatorname{Op} S^{0,*}[h]$, we say that $C=\mathcal
O_\Omega(h^{k/2})$ with some $k\in \mathbb N$ if, for
$j=0,\ldots,k-1$,
\[
c_j(x,D_x,y,\eta)=0, \quad \forall (y,\eta)\in \Omega.
\]
\end{defn}

Using the fact that the composition of the semiclassical
symbols is a local operation, we easily get that, if $A\in
\operatorname{Op} S^{0,*}[h]$ and $B=\mathcal O_\Omega(h^{k/2})$,
then $A\circ B=\mathcal O_\Omega(h^{k/2})$ and $B\circ A=\mathcal
O_\Omega(h^{k/2})$.

\subsection{Initialization}
We will use the variables $(x,y)$ introduced in
Section~\ref{s:dilation}. Our Grushin problem takes the form
 \begin{equation}\label{e:Grushin}
   \mathcal P_h (z) = \left(
   \begin{array}{cc}
   \widetilde  T^h  - b_0 - z & R_-
  \\
  R_+&0
  \end{array}
  \right)\,,
  \end{equation}
where $\widetilde T^h$ was introduced in
\eqref{widetildeT}-\eqref{hatTh1}, the operator $R_- : \mathcal
S({\mathbb R})\to \mathcal S({\mathbb R}^2)$ is given by
\begin{equation}\label{defR-}
  R_- f  (x,y) = h_0(x) f(y)\,,
\end{equation}
with $h_0(x)=\pi^{-\frac14}b_0^{-1/2} e^{-{b_0^{-1}}x^2/2}$ being
the normalized first eigenfunction of the harmonic oscillator
  $$ T=b_0^2  D_x^2 + x^2\,,$$
and the operator $R_+ : \mathcal S({\mathbb R}^2)\to \mathcal
S({\mathbb R})$ is given by
\begin{equation} \label{defR+}
  R_+ \phi (y) = \int h_0(x) \phi(x,y) dx\,.
\end{equation}

Note that $R_-$ and $R_+$ have a very simple structure, which
simplifies the analysis. But the counterpart is that, since our
considerations are perturbative near the bottom, we are obliged to
choose $\gamma_0$ to be small enough.

\subsection{Towards an inverse}
First, we will work at the level of symbols in $(y,\eta)$. We
subtract $b_0$ from $\sigma_0$ and introduce
\[
\Theta_0 (x,D_x,y,\eta) := \sigma_0(x,D_x,y,\eta) - b_0\,.
\]
We now look at the Grushin problem
\begin{equation}\label{e:Grushin0}
   \mathcal Q(x,D_x,y,\eta,z) : =
   \begin{pmatrix}
   \Theta_0(x,D_x,y,\eta) - z & R_-
  \\
  R_+&0
\end{pmatrix}.
\end{equation}

{\bf We first look at the invertibility for $(y,\eta)=(0,0)$.}\\
Put
\[
T=\sigma_0(0,0)=b_0^2 D_x^2 +x^2,
\]
which is the value of the Weyl  vector valued $h$-symbol
$\sigma_0(y,\eta)$ of $S_0$ at $(0,0)$ (see \eqref{S0} and
\eqref{sigma0}).

Consider the operator
$$
 \mathcal P ^0:= \mathcal Q(x,D_x,0,0,0)= \left(
   \begin{array}{cc}
  T-b_0 & R_-
  \\
  R_+&0
  \end{array}
  \right)
$$
on $L^2(\mathbb R)\times \mathbb C$. Its left inverse as an operator
on $L^2(\mathbb R)\times \mathbb C$ has the form
\begin{equation}\label{initgrus}
  \mathcal E^0 = \left(
   \begin{array}{cc}
   U_0 & R_-
  \\
  R_+&0
  \end{array}
  \right),
  \end{equation}
where
  \begin{equation}\label{defu0}
  U_0 \mbox{  is the regularized inverse of the harmonic oscillator } T-b_0\,.
  \end{equation}

\begin{lemma}\label{u0pseudo}
The operator $U_0$ introduced in \eqref{defu0} is a
pseudodifferential operator with symbol in $S^{-2}(\mathbb R^2)$.
\end{lemma}

\begin{proof} The projector $\Pi_0$ on the first eigenspace is
a pseudodifferential operator of order $0$ with symbol in $\mathcal
S(\mathbb R^2)$. Let us look at $T-b_0 + \Pi_0$. This is a globally
elliptic pseudodifferential operator of order $2$, which is
invertible. Hence, its inverse $U_1$ is a pseudodifferential
operator in $\operatorname{Op}S^{-2}$. It is then enough to observe
that $U_0= (I-\Pi_0) U_1$ which is also a pseudodifferential
operator with the same principal symbol.
\end{proof}

Observe the identities:
\begin{gather*}
U_0\,  \Theta_0 (0,0) + R_- R_+ = I\,,\,\quad U_0 R_-
=0\,,\, \\
R_+R_- = 1\,,\,\quad R_+ \Theta_0 (0,0)=0 \,.
  \end{gather*}

\subsection{Grushin's problem: step 2} \label{s:step2}
Now starting from our inverse of\\ $\mathcal Q(x,D_x,y,\eta,z)$ for
$(y,\eta)= (0,0)$ and $z=0$ constructed explicitly in
\eqref{initgrus}, we will construct the inverse for $(y,\eta,z)$ in
a neighborhood of $(0,0,0)$.

We can first consider
\begin{equation}\label{rem1}
\begin{split}
    {\mathcal E^0}\mathcal Q(x,D_x,y,\eta,z) & = I +  \left(
   \begin{array}{cc}
  U_0(\Theta_0(y,\eta)-\Theta_0(0,0))- z U_0 & 0
  \\
  R_+(\Theta_0(y,\eta)-\Theta_0(0,0)-z)&0
  \end{array}
  \right)\\
  & = \left(
   \begin{array}{cc}
 I+  U_0(\Theta_0(y,\eta)-\Theta_0(0,0))- z U_0 & 0
  \\
  R_+(\Theta_0(y,\eta)-\Theta_0(0,0)-z)&1
  \end{array}
  \right).
  \end{split}
  \end{equation}
(Note that $R_+\Theta_0(0,0) =0$ but we prefer to keep the
expression $\Theta_0(y,\eta)-\Theta_0(0,0)$ in the
formula).\\
The operator on the right hand side  is invertible  as an operator
in $\mathcal L (L^2(\mathbb R)\times \mathbb C)$. \\ It has the form
  $$
  \left(
   \begin{array}{cc}
 A & 0
  \\
  b &1
  \end{array}
  \right)\,,
  $$
where $A$ is a global pseudodifferential operator of degree $0$ on
$\mathcal S(\mathbb R)$
and $b : L^2(\mathbb R)\to \mathbb C$ is given by $u\mapsto \langle u, b\rangle$.\\
It is invertible if $A$ is invertible and then the left inverse
reads:
  $$
  \left(
   \begin{array}{cc}
 A ^{-1} & 0
  \\
 - b A^{-1} &1
  \end{array}
  \right)\,.
  $$

We can be more explicit by using  the pseudodifferential calculus.
For fixed $(y,\eta,z)$,
\[
A(y,\eta,z)= I + U_0(\Theta_0(y,\eta)-\Theta_0(0,0))- z U_0
\]
is a pseudodifferential operator  of order $0$ with $C^\infty$
coefficients with respect to
$y,\eta,z$.\\
It can be shown (see \cite[Section 4]{BGH}, which treats a much more
complicate case) that there exists an open neighborhood
$\Omega_1\subset \mathbb R^2$ of $(0,0)$ and $\alpha_0>0$ such that,
for any $(y, \eta)\in \Omega_1$ and $z\in \mathbb C$ such that
$|z|<\alpha_0$ the operator $A(y,\eta,z)$ is invertible as an
operator in $L^2(\mathbb R)$ and its inverse
\begin{equation}\label{e:W0}
W_0(y,\eta,z)=(I + U_0 (\Theta_0(y,\eta)-\Theta_0(0,0))- z
U_0)^{-1}\,,
\end{equation}
is also a pseudodifferential operator of order $0$ whose symbol
depends smoothly on $(y,\eta,z)$. We note that
\[
W_0(0,0,z)= (I - z U_0)^{-1}\,,
\]
and that  the left inverse of the system in the right hand side of
\eqref{rem1} takes the form
     \begin{multline*}
    \left( I +  \left(
   \begin{array}{cc}
  U_0(\Theta_0(y,\eta)-\Theta_0(0,0))- z U_0 & 0
  \\
  R_+(\Theta_0(y,\eta)-\Theta_0(0,0)-z)&0
  \end{array}
  \right)\right)^{-1}\\ = \left(
   \begin{array}{cc}
  W_0(y,\eta,z)& 0
  \\
-   R_+(\Theta_0(y,\eta)-\Theta_0(0,0)-z) W_0(y,\eta,z) &1
  \end{array}
  \right)\,.
  \end{multline*}
We can now compose this inverse with ${\mathcal E^0}$ and obtain the
inverse of the operator $\mathcal Q(x,D_x,y,\eta,z)$ for any $(y,
\eta)\in \Omega_1$ and $z\in \mathbb C$ such that $|z|<\alpha_0$ in
the block form
\begin{align*}
\mathcal E^0(x,D_x,y,\eta,z):=&\mathcal Q(x,D_x,y,\eta,z)^{-1}\\ = &
\left(
   \begin{array}{cc}
  W_0(y,\eta,z)& 0
  \\
-   R_+(\Theta_0(y,\eta)-\Theta_0(0,0)-z) W_0(y,\eta,z) &1
  \end{array}
  \right) \circ {\mathcal E_0}\\
 = & \left(
   \begin{array}{cc}
  \epsilon^0& \epsilon^0_-
  \\
\epsilon^0_+&\epsilon^0_{\pm}
  \end{array}
  \right) \,.
  \end{align*}

We observe that the first term
\begin{equation}\label{e:epsilon0}
\epsilon^0(y,\eta,z)=W_0(y,\eta,z)\, U_0
\end{equation}
is a global pseudo-differential operator of order $-2$, whose symbol
depends smoothly on $(y,\eta,z)$. \\
The second term
\begin{equation}\label{e:epsilon0-}
\epsilon^0_-(y,\eta,z)=W_0(y,\eta,z)\,R_-
\end{equation}
is an Hermite operator from $\mathcal S({\mathbb R})$ to $\mathcal
S({\mathbb R}^2)$ of the form $\Psi_- R_-$ with some $\Psi_-\in
S^{0,*}$.\\
The third term
\[
\epsilon^0_+(y,\eta,z) =R_+(1-
(\Theta_0(y,\eta)-\Theta_0(0,0)-z)W_0(y,\eta,z) U_0)
\]
is an operator from $\mathcal S({\mathbb R}^2)$ to $\mathcal
S({\mathbb R})$ of the form $R_+ \Psi_+$ with some $\Psi_+\in
S^{0,*}$. We note that for $u\in L^2(\mathbb R)$ we have:
\[
R_+ \Psi_+u= \langle u\,,\,\Psi_+^* h_0(x)\rangle_{L^2(\mathbb
R)}\,.
\]

Finally, the fourth term of the matrix is a (scalar) $C^\infty$
function of $(y,\eta,z)$:
\begin{equation}\label{e:epsilon0pm}
 \epsilon^0_{\pm} (y,\eta,z)=  -  R_+ (\Theta_0(y,\eta)-\Theta_0(0,0)-z) W_0 (y,\eta,z) R_-
\end{equation}
 which   for $(y,\eta) = (0,0)$ is equal to
 $$
\epsilon^0_{\pm} (0,0,z)= - z  R_+W_0 (0,0,z)R_-= - z R_+ (I-
zU_0)^{-1} R_- = -z.
 $$
We can write it in the form:
\begin{equation}\label{e:epsilon0pm1}
  \epsilon^0_{\pm} (y,\eta,z)=  \langle (\Theta_0(y,\eta)-\Theta_0(0,0)-z)
  W_0 (y,\eta,z) h_0\,,\, h_0\rangle_{L^2(\mathbb R)}\,.
\end{equation}

\begin{rem}\label{lemma3.2} Using the fact that $(0,0)$ is a critical point of $\hat b$,
one can show, by direct computations, that
\[
\frac{\partial \epsilon^0_{\pm}}{\partial y} (0,0,z)= \frac{\partial
\epsilon^0_{\pm}}{\partial \eta } (0,0,z)=0
\]
and
\[
\frac{\partial^2 \epsilon^0_{\pm}}{\partial y^2}
(0,0,0)=\frac{\partial^2\hat b}{\partial y^2} (0,0)\,,\quad
\frac{\partial^2\epsilon^0_{\pm}}{\partial y\partial \eta
}(0,0,0)=\frac{\partial^2\hat b}{\partial y\partial \eta }(0,0)\,,
\]
\[
\frac{\partial^2 \epsilon^0_{\pm}}{\partial \eta^2}
(0,0,0)=\frac{\partial^2 \hat b}{\partial \eta^2}(0,0)\,.
\]
\end{rem}

\begin{rem}
Note that in this subsection, we have to choose $\gamma_0$ small
enough in order to stay in a sufficiently small neighborhood of
$(0,0)$.
\end{rem}

\subsection{From one Grushin problem to another}\label{ss5}
In \cite{He77}, there is a computation of the symbol of
$\epsilon^0_\pm$. In particular it is proven that one can compute
$\epsilon^0_\pm$ for a suitable Grushin problem. The point is that
the Grushin problem considered in \cite{He77} is not the same as
above (see \eqref{e:Grushin0}). In \cite{He77}, the Grushin problem
depends on $(y,\eta)$ in the sense that one uses the eigenfunction
of the harmonic oscillator $L=\Theta_0(y,\eta) - z$ corresponding to
the eigenvalue $\hat b(y,\eta)-b_0 - z$. Hence it is necessary to
control the link between two different Grushin problems. We will use
the index $0$ for the Grushin problem defined in \eqref{e:Grushin0}
and the index $1$ for another Grushin  problem.

Thus, consider two Grushin problems
\[
\mathcal Q^0(y,\eta,z)=\begin{pmatrix} L & R^0_-
\\
R^0_+&0
\end{pmatrix}, \quad \mathcal Q^1(y,\eta,z)=\begin{pmatrix} L & R^1_-
\\
R^1_+&0
\end{pmatrix}.
\]
Let $\mathcal E^0(y,\eta,z)$ and $\mathcal E^1(y,\eta,z)$ be the
inverses of $\mathcal Q^0(y,\eta,z)$ and $\mathcal Q^1(y,\eta,z)$
respectively:
\[
\mathcal E^0=\begin{pmatrix} \epsilon^0& \epsilon^0_-\\
\epsilon^0_+&\epsilon^0_\pm\end{pmatrix}, \quad \mathcal E^1=\begin{pmatrix} \epsilon^1& \epsilon^1_-\\
\epsilon^1_+&\epsilon^1_\pm\end{pmatrix}.
\]

A relation between $\epsilon^1_\pm(y,\eta,z)$ and
$\epsilon^0_\pm(y,\eta,z)$ is given by the following lemma.

\begin{lemma}
If $R^1_+(0,0) = R^0_+(0,0)$ and $R^1_-(0,0) = R^0_-(0,0)$,  there
exists $q(y,\eta,z)$ elliptic for  $(y,\eta,z)$ close to $(0,0,0)$
such that
\begin{equation}\label{gp1}
\epsilon^0_\pm(y,\eta,z) = q(y,\eta,z) \epsilon^1_\pm(y,\eta,z) \,,
\end{equation}
and
$$
q(0,0,z) =1\,.
$$
\end{lemma}

\begin{proof}
To simplify  notation, we omit  the reference to $(y,\eta,z)$.
Computing $\mathcal E^0\mathcal Q^1\mathcal E^1$ in two different
ways, we get the identity
\[
\mathcal E^0=\mathcal E^1+\mathcal E^0(\mathcal Q^1-\mathcal
Q^0)\mathcal E^1.
\]
For the lower right entry in this matrix identity, we get
$$
\epsilon^0_\pm = \epsilon^1_\pm + \epsilon^0_\pm (R^1_+ -R^0_+)
\epsilon^1_-  + \epsilon^0_+(R^1_--R^0_-) \epsilon^1_\pm\,.
$$
Since $R^1_-(0,0) = R^0_-(0,0)$, the operator $1+
\epsilon^0_+(R^1_--R^0_-)$ is invertible for $(y,\eta)$ close to
$(0,0)$, and we obtain that
\begin{equation}\label{gp2}
(1+ \epsilon^0_+(R^1_--R^0_-))^{-1} \epsilon^0_\pm (1- (R^1_+
-R^0_+) \epsilon^1_- )  = \epsilon^1_\pm \,,
\end{equation}
that immediately completes the proof.
\end{proof}

The consequence of this lemma is that if we find easier to compute
the symbol of $\epsilon^1_\pm$ we will get the symbol of
$\epsilon^0_\pm$ up to the multiplication by an elliptic symbol.  \\
We took as $R_-^1(y,\eta)$ the operator from $\mathbb C$ into
$L^2(\mathbb R_x)$:
\begin{equation}\label{defR-yeta}
\mathbb C \ni \lambda \mapsto  R_-^1 (y,\eta)  \lambda  = \lambda
h_{y,\eta}(\cdot ) \,,
\end{equation}
$h_{y,\eta}$ being the normalized first eigenfunction of
$\sigma_0(x,D_x,y,\eta)$ associated with the eigenvalue $\hat
b(y,\eta)$ (see \eqref{defhyeta}). We took as $R_+^1(y,\eta)$ the
adjoint of $R_-^1(y,\eta)$\,.

For the inverse of $\mathcal Q^1$, a direct computation gives
\[
\mathcal E^1(y,\eta,z)=
\left(\begin{array}{cc} \epsilon^1(y,\eta,z) &  R_-^1 (y,\eta) \\
R_+^1 (y,\eta)   &  \hat b(y,\eta)-b_0-z
\end{array}
\right)\,,
\]
where $\epsilon^1(y,\eta,z)$ is the inverse of  $L= \Theta_0(y,\eta)
-z$  when restricted to the orthogonal of $h_{y,\eta}$ and $0$ on
$h_{y,\eta}$.

\begin{rem} Note here that a natural condition on $z$ is that
\begin{equation}\label{condz}
z < 2 \hat b(y,\eta)\,,
\end{equation}
in order to avoid the second eigenvalue $2 \hat b(y,\eta)$ of  $\sigma_0(x,D_x,y,\eta)-\hat b(y,\eta)$.
\end{rem}
 Hence we finally get:
\begin{prop}\label{p:E0}
There exists $q_0(y,\eta,z)$ elliptic for  $(y,\eta,z)$ close to $(0,0,0)$ such that
\begin{equation}\label{gp3}
\epsilon^0_\pm(y,\eta,z) = q_0(y,\eta,z) (\hat b(y,\eta)-b_0-z) \,,
\end{equation}
and
$$
q_0(0,0,0) =1\,.
$$
\end{prop}
In particular, we recover in another way the statements of Remark
\ref{lemma3.2}.

 \subsection{Grushin's problem: final step}\label{ss6}
 In this subsection, we will complete our study of Grushin's
problem $\mathcal P_h(z)$ given by \eqref{e:Grushin}, considering it
at the level of operators.

First, we introduce an appropriate algebra of operators. Consider
the space $\mathfrak A$ of operators on $(\mathcal S(\mathbb
R_x)\widehat \otimes C_c^\infty(\mathbb R_y)) \times
C_c^\infty(\mathbb R_y)$ of the form
\[
A(h)=\begin{pmatrix} a_0(x,D_x,y, hD_y,h) & a_-(x,D_x,y, hD_y,h)R_-\\
R_+a_+(x,D_x,y, hD_y,h) & a_\pm(y, hD_y,h),
\end{pmatrix},
\]
where $a_0,a_-,a_+\in \operatorname{Op} S^{0,*}[h]$ and $a_\pm\in
\operatorname{Op} S^{*}[h]\,$.

Here $R_+a_+(x,D_x,y, hD_y,h)$ can be defined in the following way:
$$
\mathcal S(\mathbb R_x)\widehat \otimes C_0^\infty(\mathbb R_y)  \ni
u\mapsto R_+ \, a_+(x,D_x,y, hD_y,h) u \in C^\infty (\mathbb R_y)\,,
$$
with the property that, $ \forall v \in C_0^\infty(\mathbb R_y)$,
$$
\langle R_+ \, a_+(x,D_x,y, hD_y,h) u,v\rangle = \langle
a_+(x,D_x,y, hD_y,h) u\,,\, h_0\otimes v\rangle_{L^2(\mathbb
R^2_{xy})}\,.
$$
Observe that, for any $a\in S^*$, we have
\[
(I\otimes a(y, hD_y,h)) R_-= R_-a(y, hD_y,h),
\]
and
\[
R_+(I\otimes a(y, hD_y,h))= a(y, hD_y,h)R_+\,.
\]
We also have
\[
 R_+ \theta(x,D_x,y,hD_y,h) R_- = \phi (y,hD_y,h)\,.
\]
One can see that $\mathfrak A$ is an algebra. For $A,B\in \mathfrak
A$ we have
\[
A\circ B=\begin{pmatrix} a_0\circ
b_0+a_-\circ R_-R_+\circ b_+ & (a_0\circ b_-+a_-\circ (I\otimes b_\pm))R_-\\
R_+(a_+\circ b_0+(I\otimes a_\pm)\circ b_+) & a_\pm\circ
b_\pm+R_+a_+\circ b_-R_-\,,
\end{pmatrix}.
\]
For any $A\in \mathfrak A$ of the form
\[
A(h)=\begin{pmatrix} a_0(x,D_x,y, hD_y,h) & a_-(x,D_x,y, hD_y,h)R_-\\
R_+a_+(x,D_x,y, hD_y,h) & a_\pm(y, hD_y,h)\,,
\end{pmatrix},
\]
we define the Weyl vector valued $h$-symbol of $A$ as a function on
$\mathbb R^2$ whose value at $(y,\eta)\in \mathbb R^2$ is an
operator in $L^2(\mathbb R)\times \mathbb C$ given by
\[
A(y,\eta,h)=\begin{pmatrix} a_0(x,D_x,y,\eta,h) & a_-(x,D_x,y,\eta,h)R_-\\
R_+a_+(x,D_x,y,\eta,h) & a_\pm(y,\eta,h)
\end{pmatrix}.
\]

We observe that Grushin's problems $\mathcal P_h(z)$ and $\mathcal
Q(z)$ belong to $\mathfrak A $.

We take as the first approximate inverse for $\mathcal P_h(z) $ the
operator $\mathcal E^0_\chi(z)\in \mathfrak A$, whose Weyl vector
valued $h$-symbol is
\[
\mathcal E^0_\chi(x,D_x,y,\eta,z):=\chi(y,\eta)\mathcal
E^0(x,D_x,y,\eta,z),
\]
where $\mathcal E^0(x,D_x,y,\eta,z)$ is constructed in
Section~\ref{s:step2} and $\chi\in C^\infty(\mathbb R^2)$ such that
${\rm supp}\,\chi\subset \Omega_1$ and $\chi \equiv 1$ on $\Omega$
where $\Omega$ is an open neighborhood of $(0,0)$ in $\mathbb R^2$
such that $\overline\Omega\subset \Omega_1$. Then we get
\begin{equation}\label{e:E0-left:inverse}
\mathcal E^0_\chi(z) \circ \mathcal P_h(z) = I + \Sigma (z)\,,
\end{equation}
where $\Sigma(z)\in \mathfrak A$, $\Sigma(z)=\mathcal
O_\Omega(h^{1/2})$. One can easily see that
\[
\Sigma (z)=\sum_{j=1}^{+\infty}h^{j/2}\Sigma_j(z)+ \mathcal
O_\Omega(h^{\infty}),
\]
where
\begin{equation}\label{e:Sigma}
\Sigma_j = \mathcal E^0_\chi(z)  \circ \left(\begin{array}{cc} S_j&
0\\0 &0
\end{array} \right)=
\left(\begin{array}{cc} \chi\epsilon^0 \circ S_j &0\\
\chi\epsilon^0_+ \circ S_j&0
\end{array} \right) \,.
\end{equation}

Next, we construct the inverse of $I+\Sigma(z)$ in $\mathfrak A$.
Formally, using the Neumann series, we first get that
\[
(I+\Sigma(z))^{-1} \sim I + \mathcal T(z),
\]
where
\[
\mathcal T(z)\sim \sum_{j=1}^{+\infty}(-1)^j\Sigma(z)^j
\]
with $\Sigma(z)^j=\mathcal O_\Omega(h^{j/2})$, and then the inverse
$\mathcal E_h(z)$ of $\mathcal P_h(z)$ is obtained as
\[
\mathcal E_h(z)=\left(I + \mathcal T(z) \right)\circ \mathcal
E^0_\chi(z).
\]
The problem is that the order of $\Sigma(z)^j$ as a global
differential operator in $x$ goes to $+\infty$ as $j\to \infty$, so
the sum $\mathcal T(z)$, if existing, should be of infinite order in
$(x,D_x)$. Therefore, we cut the formal expansion for $\mathcal
T(z)$, choosing some natural $N$ and putting
\[
\mathcal T_N(z) : =\sum_{j=1}^{N-1}(-1)^j\Sigma(z)^j \in \mathfrak
A.
\]
We obtain that $\mathcal T_N(z)=\mathcal O_\Omega(h^{1/2})$ and $I +
\mathcal T_N(z)$ is the inverse of $I+\Sigma(z)$ in $\mathfrak A$
modulo $\mathcal O_\Omega(h^{N/2})$:
\[
(I + \Sigma(z))\circ (I + \mathcal T_N(z))=(I + \mathcal
T_N(z))\circ (I + \Sigma(z))=I+\Sigma(z)^N = I+ \mathcal
O_\Omega(h^{N/2}).
\]
Finally, we put
\begin{equation}\label{e:ENh}
\mathcal E^N_h(z):=\left(I + \mathcal T_N(z) \right)\circ \mathcal
E^0_\chi(z)\in \mathfrak A.
\end{equation}
So we obtain that
\begin{equation}\label{e:left-inverse}
\mathcal E^N_h(z)\circ \mathcal P_h(z)=\left(I + \mathcal T_N(z)
\right)\circ (I + \Sigma (z))=I+\mathcal O_\Omega(h^{N/2}).
\end{equation}
Thus, we have found a left inverse $\mathcal E^N_h(z)$ for the
operator $\mathcal P_h(z)$ in the algebra $\mathfrak A$ modulo
$\mathcal O_\Omega(h^{N/2})$.

Similarly, we can construct a right inverse for the operator
$\mathcal P_h(z)$ in the algebra $\mathfrak A$ modulo $\mathcal
O_\Omega(h^{N/2})$, which implies that $\mathcal E^N_h(z)$ is a
two-sided inverse for the operator $\mathcal P_h(z)$ in the algebra
$\mathfrak A$ modulo $\mathcal O_\Omega(h^{N/2})$:
\begin{equation}\label{e:right-inverse}
\mathcal P_h(z)\circ \mathcal E^N_h(z)=I+\mathcal
O_\Omega(h^{N/2})\,.
\end{equation}
Denote
\begin{equation}\label{forminv}
\mathcal E^N_h(z)= \left(
   \begin{array}{cc}
  \epsilon^N(z) & \epsilon^N_-(z)
  \\
\epsilon^N_+(z)&\epsilon^N_{\pm}(z)
  \end{array}
  \right)\,.
\end{equation}

If we write
\[
\mathcal E^N_h(z)\sim \mathcal E^N_0(z)+h^{\frac 12}\mathcal
E^N_1(z)+h \mathcal E^N_2(z) +\ldots,
\]
then we have
\[
\mathcal E^N_j(y,\eta,z)=\mathcal E_j(y,\eta,z), \quad (y,\eta)\in
\Omega,\quad j=1,2,\ldots,N-1,
\]
where $\mathcal E_j(z)$ are the coefficients of the formal
expansion:
\begin{equation}\label{e:Ehformal}
\mathcal E_h(z)\sim \mathcal E_0(z)+h^{\frac 12}\mathcal E_1(z)+h
\mathcal E_2(z) +\ldots.
\end{equation}
Writing the formal expansion $\mathcal E_h(z)$ in a block form:
\begin{equation}\label{blo1}
\mathcal E_h(z)= \left(
   \begin{array}{cc}
  \epsilon (z) & \epsilon_-(z)
  \\
\epsilon_+(z)&\epsilon_{\pm}(z)
  \end{array}
  \right)\,,
\end{equation}
we observe that $\epsilon_{\pm}(z)$ can be considered not only as a
formal series, but also as an $h$-pseudodifferential operator on
$L^2(\mathbb R)$:
\begin{equation}\label{blo2}
\epsilon_{\pm}(z)\sim \sum_{j=0}^\infty
h^{\frac{j}{2}}\epsilon_{j\pm}(y, hD_y,z).
\end{equation}
We have
\begin{equation}\label{e:epsilon-diff}
\epsilon_{\pm}(y, hD_y,z)=\epsilon^N_{\pm}(y, hD_y,z)+ \mathcal
O_\Omega(h^{N/2}),
\end{equation}
or equivalently
\[
\epsilon_{j\pm}(y,\eta,z)=\epsilon^N_{j\pm}(y,\eta,z), \quad
(y,\eta)\in \Omega,\quad j=1,2,\ldots,N-1.
\]

Let us show how to compute the first two coefficients of
\eqref{e:Ehformal}.
\medskip\par
{\bf The coefficient of $h^\frac 12$} is given by
\[
\mathcal E_1=-\Sigma_1\circ \mathcal E^0_\chi.
\]
Using \eqref{e:Sigma}, we get
\begin{align*}
\mathcal E_1 & = - \left(\begin{array}{cc} \chi\epsilon^0\circ S_1 &0\\
\chi\epsilon^0_+ \circ S_1 &0
\end{array}
\right) \circ \left(
\begin{array}{cc} \chi\epsilon^0 & \chi\epsilon^0_- \\
\chi\epsilon^0_+ & \chi\epsilon^0_\pm
   \end{array}
   \right)\\ & = -\left(
\begin{array}{cc} \chi\epsilon^0\circ S_1 \circ \chi\epsilon^0 &
\chi\epsilon^0\circ S_1 \circ \chi\epsilon^0_- \\
\chi\epsilon^0_+ \circ S_1 \circ \chi\epsilon^0 & \chi\epsilon^0_+
\circ S_1 \circ \chi\epsilon^0_-
   \end{array}
   \right) \,.
\end{align*}

So the correction $\epsilon_{1\pm}(y, hD_y,z)$ is a
$h$-pseudodifferential operator given by:
\[
\epsilon_{1\pm}= -R_+\chi(1- (\Theta_0 -\Theta_0(0,0)-z)W_0 U_0)
\circ S_1\circ \chi W_0 R_-\,.
\]
Since the operators $\Theta_0$, $W_0$ and $U_0$ respect the parity
and $S_1$ changes the parity in $x$, we obtain\footnote{This type of
argument appears in Sj\"ostrand \cite{Sj} who refers to Grushin
\cite{Gru}, and then  in the paper of B. Helffer \cite{He77} devoted
to the hypoellipticity with loss of $\frac 32$ derivatives.} that
\[
\epsilon_{1\pm}(y, hD_y,z)=0.
\]

{\bf The coefficient of $h$} is given by
\[
\mathcal E_2 =(\Sigma^2_1 -\Sigma_2 )\circ \mathcal E^0_\chi.
\]
Using \eqref{e:Sigma}, we get
\[
\mathcal E_2 = \left(\begin{array}{cc} \chi\epsilon^0  \circ
(S_1\circ \chi\epsilon^0
\circ S_1-S_2) &0\\
\chi\epsilon^0_+   \circ (S_1 \circ \chi\epsilon^0  \circ S_1 - S_2)
&0
\end{array} \right)
\circ \left(
\begin{array}{cc} \chi\epsilon^0 & \chi\epsilon^0_- \\
\chi\epsilon^0_+ & \chi\epsilon^0_\pm
   \end{array}
   \right)\,.
\]

The correction $\epsilon_{2\pm}(y, hD_y,z)$ is given by
\[
\epsilon_{2\pm}=\chi\epsilon^0_+ \circ (S_1 \circ \chi\epsilon^0
\circ S_1 - S_2)\circ \chi\epsilon^0_-\,.
\]

It looks rather difficult to compute this coefficient explicitly.
But of course, this is just a rather routine long computation. If we
are interested in the low lying eigenvalues, the approach used in
our previous work \cite{HKI} is better. Hence we will not pursue in
this direction.

\section{Grushin's problem and quasimodes}\label{s:Grushin-quasimodes}

By the Grushin method, we formally arrive at a statement of the type
$z\in \sigma(\widetilde T_h)$ is equivalent to $0\in \sigma(
\epsilon_{\pm}(y, h D_y,z;h))$. This kind of problem is treated in
\cite{HeSjharp3}, where the notion of $\mu$-spectrum is introduced
(see Definition 3.2).

Let us choose $\gamma_0 >0$ such that $\gamma_0<\alpha_0$ and the
set $\{(y,\eta)\in \mathbb R^2: \hat b(y,\eta)<b_0+\gamma_0\}$ is
connected and contained in $\Omega$.

By Proposition~\ref{p:E0}, we have
\[
\epsilon_{0\pm}(y,\eta,z) = q_0(y,\eta,z) (\hat b(y,\eta)-b_0-z) \,,
\]
where $q_0(y,\eta,z)$ is elliptic for $(y,\eta)\in \Omega$ and
$|z|<\alpha_0$. Let us extend $q_0$ to an elliptic semiclassical
symbol from $S(1)$. The operator $Q_0(z)=q_0(y,hD_y,z)$ is
invertible as an operator in $L^2(\mathbb R)$ and the inverse
$Q_0(z)^{-1}$ is an elliptic $h$-pseudodifferential operator.
Consider an $h$-pseudodifferential operator $p_{\rm eff}(z)=p_{\rm
eff}(y, hD_y,h, z)$ given by
\[
p_{\rm eff}(z)=Q_0(z)^{-1}\circ \epsilon_\pm (z).
\]
Then we have
\[
p_{\rm eff}(y, \eta ,h, z) \sim\sum_{j\in \mathbb N} p^j_{\rm
eff}(y,\eta,z) h^j\,,
\]
with
\[
p^0_{\rm eff}(y,\eta,z) = \hat b(y,\eta) - b_0 -z
\]
for $(y,\eta)\in \Omega$ and $|z|<\alpha_0$.

\subsection{The direct statement}\label{appendixD} Here
we follow Helffer-Sj\"ostrand (\cite{HeSjharp1,HeSjharp3}) for the
$1D$-problem and Fournais-Helffer \cite{FoHe}.

Choose $\gamma_0(h)\in (b_0,\alpha_0)$ defined for $h\in (0,h_0]$
such that $\gamma_0(h)\to \gamma_0$ as $h\to 0$ and there exists
$a(h)>0, h\in (0,h_0]$ such that $a(h)=\mathcal O(h^{N_0})$ and
\begin{equation}\label{e:gap}
\sigma(H^h)\cap (h(b_0+\gamma_0(h)),
h(b_0+\gamma_0(h)+a(h)))=\emptyset\,.
\end{equation}

Suppose that we have found $z=z(h)$, satisfying $|z(h)|<
\gamma_0(h)$ for $h\in (0,h_0]$ and the corresponding approximate
0-eigenfunction $u_h^{qm}\in C^\infty(\mathbb R)$ of the operator
$p_{\rm eff}(y, hD_y,h, z(h))$, i.e.
\[
p_{\rm eff}(z(h)) u_h^{qm} = \mathcal O (h^\infty)\,,
\]
such that $\|u_h^{qm}\|=1+\mathcal O (h^\infty)$ and the frequency
set of $u_h^{qm}$ is non-empty and contained in $\Omega\,$. Then
$u_h^{qm}$ is the approximate 0-eigenfunction of the operator
$\epsilon_\pm (z)$:
\[
\epsilon_\pm (z) u_h^{qm} = Q_0(z)(p_{\rm eff}(z) u_h^{qm}) =
\mathcal O (h^\infty)\,.
\]

Define the function $\psi_h\in \mathcal S^2(\mathbb R^2)$ by
\[
\psi_h=\epsilon^N_{-}(z) u_h^{qm}, \quad h\in (0,h_0].
\]
Using the fact that $\mathcal E^N_h(z)$ is the right inverse for
$\mathcal P_h(z)$ in $\mathfrak A$ modulo $\mathcal
O_\Omega(h^{N/2})$, by \eqref{e:right-inverse}, we obtain that
\[
(\widetilde {T}_h -b_0- z(h)) \epsilon^N_- (z)+ R_- \epsilon^N_{\pm}
(z)=K_-\in \mathcal O_\Omega(h^{N/2})\,.
\]
We get
\[
(\widetilde {T}_h -b_0- z(h)) \psi_h + R_- \epsilon^N_{\pm}
(z)u_h^{qm}=K_-u_h^{qm}\,.
\]
Since the frequency set of $u_h^{qm}$ is contained in $\Omega$ and
$K_-\in \mathcal O_\Omega(h^{N/2})$, we have
\[
K_-u_h^{qm}= \mathcal O(h^{N/2}).
\]
By \eqref{e:epsilon-diff} and the fact that the frequency set of
$u_h^{qm}$ is contained in $\Omega$, we also have
\[
\epsilon^N_{\pm} (z)u_h^{qm}=\epsilon_{\pm} (z)u_h^{qm}+\mathcal
O(h^{N/2})=\mathcal O(h^{N/2}).
\]
So we arrive at
\[
(\widetilde {T}_h -h^{-1}\mu_h) \psi_h= \mathcal O(h^{N/2})\,,
\]
where $\mu_h=h(b_0+z(h))\in [hb_0,h(b_0+\gamma_0(h)]$.

To control the norm of $\psi_h$, using \eqref{e:ENh} and the fact
that the frequency set of $u_h^{qm}$ is contained in $\Omega$, we
observe that
\[
\|\psi_h\|=\|\epsilon^0_- (z)u_h^{qm}\|(1+\mathcal O(h^{1/2})).
\]
Then we see that
\[
\epsilon^0_- (z)=W_0 R_-=(I + U_0\circ
(\Theta_0(y,hD_y)-\Theta_0(0,0))- z U_0)^{-1}R_-.
\]
Therefore
\[
\|\psi_h\|=\|\epsilon^0_- (z)u_h^{qm}\|\geq
C\|R_-u_h^{qm}\|=C\|u_h^{qm}\|=C(1+\mathcal O (h^\infty)).
\]
Put
\[
\widetilde v_h=\frac{\psi_h}{\|\psi_h\|}.
\]
Then we have $\|\widetilde v_h\|=1$ and
\[
(\widetilde {T}_h -h^{-1}\mu_h) \widetilde v_h=\mathcal
O(h^{N/2}).
\]
Coming back to the initial variables (see
Subsection~\ref{s:frequency}), we obtain a function $v_h$ such that
$\|v_h\|=1$ and
\[
(H^h-\mu_h)v_h=\mathcal O(h^{N/2+1})\,.
\]
Since $N$ is arbitrary, by Spectral Theorem, for any $h\in (0,h_0]$,
there exists $\lambda_h \in \sigma (H^h)\cap [hb_0, h (b_0
+\gamma_0(h)))$ such that
\[
\lambda_h-\mu_h=\lambda_h-h(b_0+z(h))=\mathcal O(h^\infty).
\]
By \eqref{e:gap}, it follows that $\lambda_h\in
[hb_0,h(b_0+\gamma_0(h)]$.

\subsection{The converse statement}
This time we start from an $L^2$ eigenfunction $u_h$ of $H^h$
associated with $\lambda_h\in [hb_0,h(b_0 +\gamma_0)]$ for any $h\in
(0,h_0]$ with $\gamma_0>0$ as above. The aim is to construct an
approximate 0-eigenfunction for the operator $p_{\rm eff}(z)$ with
$z(h)=\frac{1}{h}(\lambda_h-hb_0)$.

Performing metaplectic transformations as in
Subsection~\ref{s:metaplectic}, we arrive at an $L^2$ eigenfunction
$\widetilde u_h$ of $\widetilde{T}_h$ associated with
$h^{-1}\lambda_h$. Now we fix some natural $N$ and use the fact that
the operator $\mathcal E^N_h(z)\in \mathfrak A$ is the left inverse
for $\mathcal P_h(z)$ in $\mathfrak A$ modulo $\mathcal
O_\Omega(h^{N/2})$. In particular, \eqref{e:left-inverse} reads:
\begin{equation}\label{E.2}
\epsilon^N_+(z) (\widetilde{T}_h-h^{-1}\lambda_h) +
\epsilon^N_\pm(z) R_+=K_+\,,
\end{equation}
where $K_+=R_+a_+(x,D_x,y, hD_y,h)$ with $a_+\in \operatorname{Op}
S^{0,*}[h]$, $a_+ = \mathcal O_\Omega(h^{N/2})$.

By definition, we can write
\[
a_+(x,D_x,y,hD_y,h)= A_{N-1}(h)+R_{N-1}(h),
\]
where
\[
A_{N-1}(h)=\sum_{j=0}^{N-1} h^{\frac j2}a_j(x,D_x,y,hD_y),
\]
each $a_j$ belongs to $S^{0,m+j} (\mathbb R^2\times\mathbb R^2)\,$
and can be represented as a finite sum
\[
a_j(x,D_x,y,hD_y)=\sum  b^{(j)}_\ell (x,D_x) c^{(j)}_\ell (y,hD_y)
\]
with some $b^{(j)}_\ell\in S^{m+j}(\mathbb R)$ and $c^{(j)}_\ell\in
S(1)$, and $R_{N-1}(h)$ satisfies \eqref{e:RNh}. Moreover,
$c^{(j)}_\ell(y,\eta)=0$ for any $(y,\eta)\in\Omega$ and
$j=0,1,\ldots,N-1$.

We know that $\widetilde u_h=S^h\hat u_h$, where $S^h$ is a unitary
operator in $L^2(\mathbb R^2)$ given by
$S^hf(x,y)=h^{1/4}f(h^{1/2}x,y)$. It is easy to see that
\[
c^{(j)}_\ell (y,hD_y)\circ S^h = S^h\circ c^{(j)}_\ell (y,hD_y).
\]
Therefore, we have
\[
R_+A_{N-1}(h)\widetilde  u_h=\sum R_+  b^{(j)}_\ell (x,D_x)S^h
c^{(j)}_\ell (y,hD_y) \hat u_h\,.
\]
Since the frequency set of $\hat  u_h$ is contained in $\Omega$, we
have
\[
c^{(j)}_\ell (y,hD_y) \hat u_h=\mathcal O(h^{N/2}).
\]
Since the operator $R_+ b^{(j)}_\ell (x,D_x)S^h$ is uniformly
bounded in $h$ as as operator from $L^2(\mathbb R^2)$ to
$L^2(\mathbb R)$, we obtain that
\begin{equation}
\label{e:AN} R_+A_{N-1}(h)\widetilde  u_h=\mathcal O(h^{N/2}).
\end{equation}
Using \eqref{e:RNh} and \eqref{control1}, we conclude that
\begin{equation}
\label{e:RN} R_{N-1}(h)\widetilde  u_h=\mathcal O(h^{N/2}).
\end{equation}
By \eqref{e:AN} and \eqref{e:RN}, we obtain that
\begin{equation}
\label{e:K+} K_+\widetilde  u_h=\mathcal O(h^{N/2}).
\end{equation}
By \eqref{E.2} and \eqref{e:K+}, we have
\begin{equation}
\label{e:R+} \epsilon^N_\pm(z) R_+ \widetilde  u_h= \mathcal
O(h^{N/2}).
\end{equation}
It remains to show that $R_+ \widetilde  u_h$ is not too small in
order to get effectively a quasi-mode. For this,  we need first the
following proposition:
\begin{prop}\label{p:frequency}
Let $v_h, h\in (0,h_0],$ be a family of functions in $L^2(\mathbb
R^2)$ such that $\|v_h\|= 1+\mathcal O(h)$. Let
\[
\Phi_h(y):= h^{-1/4} \int_{-\infty}^{+\infty} e^{-
\frac{x^2}{2b_0h}} \, v_h (x,y) dx\,.
\]
Then the frequency set of $\Phi_h$ is contained in the set of
$(y,\eta)$ such that $(0,0,y,\eta)$ belongs to the frequency set of
$v_h$.
\end{prop}

\begin{proof}
Fix $(y_0,\eta_0)\in \mathbb R^2$. Suppose that $\chi\in
C^\infty_c(\mathbb R^2)$ be such that $\chi\equiv 1$ in a
neighborhood of $(0,y_0)$. Then for any $y$ in some neighborhood of
$y_0$ we have as $h\to 0$
\[
\int_{-\infty}^{+\infty} e^{- \frac{x^2}{2b_0h}} \, v_h (x,y)
dx=\int_{-\infty}^{+\infty} e^{- \frac{x^2}{2b_0h}} \, \chi(x,y) v_h
(x,y) dx+ \mathcal O(h^\infty).
\]
Therefore, without loss of generality, we can assume that $v_h$ is
supported in a (arbitrary small) neighborhood of $(0,y_0)$. Then
$\Phi_h$ is supported in a neighborhood of $y_0$ and we have the
formula
\[
\langle e^{-ih^{-1}y\eta},\Phi_h(y)\rangle_y
=\int_{-\infty}^{+\infty} e^{- \frac{b_0\xi^2}{2h}} \,\langle
e^{-ih^{-1}(x\xi+y\eta)}, v_h (x,y)\rangle_{x,y} d\xi.
\]
Thus if $(0,0,y_0,\eta_0)$ is not in the frequency set of $v_h$,
then, by definition, there exist $\varepsilon_0>0$ and a
neighborhood $V_{\eta_0}$ of $\eta_0$ in $\mathbb R$ such that
\[
\langle e^{-ih^{-1}(x\xi+y\eta)}, v_h (x,y)\rangle_{x,y}=\mathcal
O(h^\infty), \quad h\to 0,
\]
uniformly on $\xi$, $|\xi|<\varepsilon_0$ and $\eta\in V_{\eta_0}$,
that immediately implies  (observing that for $|\xi|\geq \epsilon_0$
the contribution is exponentially small) that
\[
\langle e^{-ih^{-1}y\eta},\Phi_h(y)\rangle_y =\mathcal O(h^\infty),
\quad h\to 0,
\]
uniformly on $\eta\in V_{\eta_0}$, and, therefore, $(y_0,\eta_0)$ is
not in the frequency set of $\Phi_h$.
\end{proof}

Let us rewrite the formula
\[
R_+ \widetilde u_h(y)=\pi^{-1/4}b_0^{-1/2} \int_{-\infty}^{+\infty}
e^{-\frac{x^2}{2b_0}} \widetilde u_h(x,y)dx
\]
in the form
\[
R_+ \widetilde u_h(y)=\pi^{-1/4}b_0^{-1/2}
h^{-1/4}\int_{-\infty}^{+\infty} e^{-\frac{x^2}{2b_0h}} \widehat
u_h(x,y)dx,
\]
which corresponds simply to the change of variable $x=h^{-1/2}\tilde
x$ in the integral. Here we note that $$ \tilde u_h (x,y) = h^{\frac
14} \widehat u_h (h^{\frac 12} x,y)\,.$$

Now we apply Proposition~\ref{p:frequency} with $v_h=\widehat u_h$.
By \eqref{e:freq-set}, we know that the frequency set of $v_h$ is
contained in $ \{(x,y,\xi,\eta)\in T^*{\mathbb R}^2 \,|\,
(x,\xi)=(0,0),\, \hat b(y,\eta)<b_0+\gamma_0\}$.
Proposition~\ref{p:frequency} implies that the frequency set of $R_+
\widetilde u_h$ is contained in $ \{(y,\eta)\in T^*{\mathbb R} \,|\,
\hat b(y,\eta)<b_0+\gamma_0\}$. Using this fact,
\eqref{e:epsilon-diff} and \eqref{e:R+}, we obtain that
\[
\epsilon_\pm(z) R_+ \widetilde  u_h=\epsilon^N_\pm(z) R_+ \widetilde
u_h+\mathcal O(h^{N/2})=\mathcal O(h^{N/2}).
\]

To control the norm of $R_+ \widetilde  u_h\,$, we use
\eqref{e:E0-left:inverse} and write
\[
\epsilon^0(z) (\widetilde{T}_h-h^{-1}\lambda_h) + \epsilon^0_-(z)
R_+=I+\Sigma^0\,,
\]
where $\Sigma^0\in \operatorname{Op} S^{0,*}[h]$, $\Sigma^0 =
\mathcal O_\Omega(h^{1/2})$. Applying this identity to $\widetilde
u_h$, we obtain that
\[
\epsilon^0_-(z) R_+\widetilde u_h=\widetilde u_h+\Sigma^0\widetilde
u_h\,.
\]
As above (see the proof of \eqref{e:K+}), one can show that
\[
\Sigma^0\widetilde u_h=\mathcal O(h^{1/2}).
\]
Therefore, we have
\[
\|\epsilon^0_-(z) R_+\widetilde u_h\|=1+\mathcal O(h^{1/2})\,.
\]
Since $\epsilon^0_-(z)$ is bounded  as an operator from
$L^2(\mathbb R)$ into $L^2(\mathbb R^2)$ uniformly on $h\in (0,1]$,
this implies that there exist $h_0>0$ and $C>0$ such that, for any
$h\in (0,h_0]$,
\[
\|R_+\widetilde u_h\|> C\,.
\]
Thus, the function
\[
u_h^{qm}=\frac{R_+ \widetilde  u_h}{\|R_+ \widetilde  u_h\|}
\]
is the approximate $L^2$ normalized eigenfunction for
$\epsilon_{\pm} (z)$:
\[
\epsilon_\pm(z) u_h^{qm}= \mathcal O(h^{N/2}), \quad \|u_h^{qm}\|=1\,.
\]
Since $N$ is arbitrary, we obtain that
\[
\epsilon_\pm(z) u_h^{qm}= \mathcal O(h^{\infty}),
\]
and
\[
p_{\rm eff}(z) u_h^{qm}=Q_0(z)^{-1}(\epsilon_\pm(z) u_h^{qm})=
\mathcal O(h^{\infty})\,.
\]

\section{Concluding remarks}\label{s7}
In the treatment of the Grushin problem in Section~\ref{s:Grushin},
one can distinguish two parts. The first part is a purely formal
computation of vector-valued $h$-symbols in $(y,\eta)$ at the level
of complete formal expansions in powers of $h$ where we use the
formal Weyl composition law $\#_h$. The second part is a way of
associating with these symbols well defined operators on a Hilbert
space. This forces us to introduce a cut-off function $\chi$ and to
consider finite sums (choice of $N$) instead of formal infinite
sums. It is only in the last part that the choice of the Grushin
problem with $R_-$ and $R_+$ is useful.

If we are interested in the explicit computation of symbols, it is
better, in the spirit of what we have done in Subsection \ref{ss5},
to consider instead the Grushin problem associated with the pair
$(R_-^1, R_+^1)$ introduced in \eqref{defR-yeta}. Now if we make the
construction done in Subsection \ref{ss6} with this pair (taking
$\chi=1$, $N=+\infty$ and using the formal Weyl composition law
$\#_h$ instead of the composition of operators $\circ$), we will
arrive at an explicit formal inverse $\mathcal
E^{(1),\infty}(y,\eta;z)$ in the same form as in \eqref{blo1},
defined for any $(y,\eta)$ in $\Omega_{\eta_0}:=\{ (y,\eta)\,,\,
\hat b(y,\eta) < \eta_0\}$,
\[
\mathcal E_h^{(1),\infty}(y,\eta;z)= \left(
   \begin{array}{cc}
  \epsilon^{(1),\infty} (y,\eta;z) & \epsilon^{(1),\infty}_-(y,\eta;z)
  \\
\epsilon_+^{(1),\infty}(y,\eta;z)&\epsilon_{\pm}^{(1),\infty}(y,\eta;z)
  \end{array}
  \right)\,,
\]
where $z$ should satisfy (see \eqref{condz}):
\[
%\label{condz2}
z \in I:= [0, \inf ( \eta_0, 2 b_0))\,,
\]
and assuming that $\hat b$ is $C^\infty$ in $\Omega_{\eta_0}$.

To relate $\epsilon_{\pm}^{(1),\infty}(y,\eta;z)$ with
$\epsilon_{\pm}(y,\eta;z)$ defined for $(y,\eta)\in \Omega$ in
\eqref{blo1}, we can proceed in the same spirit as in Subsection
\ref{ss5}, but this time working with the formal vector valued
$h$-symbols defined for $(y,\eta)\in \Omega$ and the formal Weyl
composition law $\#_h$. We get, for $(y,\eta)\in \Omega$ (cf.
\eqref{gp2}):
\[
(1+ \epsilon_+ \#_h (R^1_--R_-))^{-1} \#_h \epsilon_\pm \#_h (1-
(R^1_+ -R_+) \#_h \epsilon^{(1),\infty}_- )  =
\epsilon^{(1),\infty}_\pm \,.
\]
Now, the invertibility of $1+ \epsilon_+ \#_h (R^1_--R_-)$ is just a
question of looking at the principal symbol and this is exactly what
was done in Subsection \ref{ss5}.

Note that $\epsilon^{(1),\infty}_\pm(y,\eta,z)$ is defined for any
$(y,\eta,z)\in \Omega_{\eta_0}\times I$, not necessarily close to
$(0,0,0)$. This symbol can be theoretically computed following the
composition law of the symbols. Practically this remains to be
difficult and the computation of the subprincipal symbol of
$\epsilon^{(1),\infty}_\pm$ should involve the computation of $S_2$
at the end of Subsection \ref{s:dilation}.

We think that $\epsilon^{(1),\infty}_\pm$ is the right effective
Hamiltonian which could permit us to analyze the spectrum in $[hb_0,
h \inf (b_0+\eta_0, 3 b_0))$ like in \cite{RV} but we are at the
moment obliged to go back to the effective Hamiltonian
$\epsilon_{\pm}$ for technical reasons and hence have limited our
statements to the bottom of the spectrum.

\subsection*{Acknowledgement} The first author would like to thank N.
Raymond and S. Vu Ngoc for useful discussions around various
versions of \cite{RV}. These discussions were motivating for
improving preliminary versions of our paper. The second author would
also like to thank M. Karasev for useful remarks.

\end{document}